\documentclass[11pt,twoside]{amsart}

\usepackage[T1]{fontenc}
\usepackage{lmodern}
\usepackage{microtype}
\usepackage{amsmath,amssymb,amsthm,mathtools}
\usepackage[inline]{enumitem}
\usepackage{booktabs}
\usepackage{array}
\usepackage{tikz}
\usetikzlibrary{arrows.meta,positioning}
\usepackage{xcolor}
\definecolor{insertpurple}{RGB}{112,48,160}
\usepackage[hidelinks]{hyperref}
\usepackage{xr-hyper}
\usepackage[nameinlink,capitalise]{cleveref}

\calclayout

\title[Lie algebras in symmetric tensor categories]
{Lie algebras in symmetric tensor categories}

\author{I. Angiono}
\address{Angiono: FaMAF-CIEM (CONICET), Universidad Nacional de C\'ordoba, Medina Allende s/n, Ciudad Universitaria, 5000 C\'ordoba, Argentina}
\email{ivan.angiono@unc.edu.ar}

\author{A. Czenky}\address{Czenky: Department of Mathematics, University of Southern California,
3620 S. Vermont Ave., Los Angeles, CA 90089, USA}
\email{czenky@usc.edu}

\author{P. Etingof}\address{Etingof: Department of Mathematics, Massachusetts Institute of Technology,
77 Massachusetts Avenue, Cambridge, MA 02139, USA}
\email{etingof@math.mit.edu}

\author{V. Ostrik}\address{Ostrik: Department of Mathematics, University of Oregon,
Eugene, OR 97403, USA}
\email{vostrik@uoregon.edu}

\author{J. Plavnik}\address{Plavnik: Department of Mathematics, Indiana University, Bloomington, IN 47405, USA, \& Department of Mathematics
and Data Science, Vrije Universiteit Brussel, Pleinlaan 2, 1050 Brussels, Belgium}
\email{jplavnik@iu.edu}

\author{G. Sanmarco}\address{Sanmarco: Iliad}
\email{sanmarcoguille@gmail.com}


\newtheorem{theorem}{Theorem}[section]
\newtheorem{proposition}[theorem]{Proposition}
\newtheorem{lemma}[theorem]{Lemma}
\newtheorem{corollary}[theorem]{Corollary}
\theoremstyle{definition}
\newtheorem{definition}[theorem]{Definition}

\newtheorem{question}[theorem]{Question}
\theoremstyle{remark}
\newtheorem{remark}[theorem]{Remark}

\numberwithin{equation}{section}

\newcommand{\C}{\mathcal C}
\newcommand{\W}{\mathcal W}
\newcommand{\Ocat}{\mathcal O}
\newcommand{\Pwt}{\mathsf P}
\newcommand{\one}{\mathbf 1}
\newcommand{\g}{\mathfrak g}
\newcommand{\h}{\mathfrak h}
\newcommand{\nplus}{\mathfrak n^+}
\newcommand{\nminus}{\mathfrak n^-}

\newcommand{\aLie}{\mathfrak a}
\newcommand{\bLie}{\mathfrak b}
\newcommand{\Der}{\operatorname{Der}}
\newcommand{\End}{\operatorname{End}}
\newcommand{\Hom}{\operatorname{Hom}}
\newcommand{\Irr}{\operatorname{Irr}}

\newcommand{\Rep}{\operatorname{Rep}}
\newcommand{\gr}{\operatorname{gr}}
\newcommand{\tr}{\operatorname{tr}}
\newcommand{\ad}{\operatorname{ad}}
\newcommand{\HC}{\operatorname{HC}}
\newcommand{\CE}{\mathrm{CE}}
\newcommand{\Verp}{\operatorname{Ver}_p}
\newcommand{\Verplus}{\operatorname{Ver}_p^+}
\newcommand{\id}{\operatorname{id}}

\newcommand{\Ver}{\operatorname{Ver}}

\newcommand{\he}{\mathrm{ht}}

\newcommand{\fb}{\mathfrak{b}}
\newcommand{\fg}{\mathfrak{g}}
\newcommand{\fh}{\mathfrak{h}}

\newcommand{\fn}{\mathfrak{n}}

\newcommand{\ttb}{\mathtt{b}}
\newcommand{\ttd}{\mathtt{d}}
\newcommand{\tte}{\mathtt{e}}
\newcommand{\ttf}{\mathtt{f}}

\newcommand{\tth}{\mathtt{h}}

\newcommand{\ttq}{\mathtt{q}}

\newcommand{\ttm}{\mathtt{m}}

\newcommand{\ttB}{\mathtt{B}}
\newcommand{\ttI}{\mathtt{I}}
\newcommand{\ttK}{\mathtt{K}}

\newcommand{\ttU}{\mathtt{U}}
\newcommand{\ttV}{\mathtt{V}}

\newcommand{\ttY}{\mathtt{Y}}
\newcommand{\ttW}{\mathtt{W}}

\newcommand{\ev}{\mathsf{ev}}

\newcommand{\ima}{\operatorname{im}}

\newcommand{\indcat}[1]{#1^{\operatorname{ind}}}

\newcommand{\bQ}{\mathbb{Q}}
\newcommand{\bZ}{\mathbb{Z}}
\newcommand{\bN}{\mathbb{N}}

\newcommand{\cC}{\mathcal{C}}

\newcommand{\cG}{\mathcal{G}}

\newcommand{\cO}{\mathcal{O}}

\newcommand{\cW}{\mathcal{W}}

\newcommand{\cR}{\mathcal{R}}

\newcommand{\cX}{\mathcal{X}}

\newcommand{\fc}{\mathfrak{c}}

\newcommand{\Ann}{\operatorname{Ann}}
\newcommand{\re}{\operatorname{re}}

\newcommand{\FLie}{\operatorname{FLie}}

\newcommand{\Lie}{\mathsf{Lie}}

\newcommand\restr[2]{{\left.\kern-\nulldelimiterspace #1 \vphantom{\big|} \right|_{#2}}} 

\begin{document}

\begin{abstract}
We develop the general theory of Lie
algebras in symmetric tensor categories. We prove categorical versions of classical results such as Engel's theorem, Cartan's criterion, and that all derivations of a strongly
semisimple Lie algebra are inner.  
For finite-dimensional graded Lie algebras, we construct a graded category $\Ocat$, classify its graded-simple
objects, and prove the generic graded irreducibility of Verma modules. 
Lastly, we  develop the theory of root systems for contragredient Lie
algebras. 
\end{abstract}

\maketitle

\tableofcontents

\section{Introduction}\label{sec:introduction}

The guiding question of this paper is how much of the architecture of classical Lie theory survives in  symmetric tensor categories over fields of arbitrary characteristic.
Many classical arguments depend on properties of vector spaces that need not persist in this setting. For instance, in positive characteristic, a skew-symmetric bracket satisfying the Jacobi identity need not have the PBW property \cite{EtingofPBW}. Thus, extending Lie theory from vector spaces to symmetric tensor categories requires more than a formal translation.

There is a particular reason to ask this question now. Over a field of
characteristic zero, Deligne proved that every pre-Tannakian category of
moderate growth admits a fiber functor to supervector spaces
\cite{DeligneTensor}, and hence can be reconstructed from ordinary or super
group theory. The direct analogue fails in characteristic $p$. The
fundamental counterexamples are the Verlinde categories $\Verp$, obtained by
semisimplifying the category of tilting modules for $SL_2$; for $p\geq 5$, they admit no
symmetric tensor functor to either $\operatorname{Vec}$ or
$\operatorname{sVec}$. Moreover, $\Verp$ plays an important role:
Ostrik proved that every symmetric fusion category in
characteristic $p$ admits a symmetric tensor functor to $\Verp$
\cite{OstrikSymmetricFusion}, and Coulembier, Etingof, and Ostrik subsequently extended this result to
semisimple pre-Tannakian categories of moderate growth
\cite{CEOFrobenius}. 

The study of algebraic structures in symmetric tensor categories
has expanded considerably in recent years. Much of the recent work has
focused on the Verlinde categories, including affine group schemes and
Harish--Chandra pairs \cite{VenkateshHC,Ma}, representations of
general linear groups \cite{Venkatesh,Kannan,Utiralova}, categorical
Chevalley--Eilenberg cohomology \cite{EtingofHu}, and the construction and
classification of Lie algebra objects
\cite{APS-ss,EKOweak,Newton}. Contragredient Lie algebras have also been
defined in arbitrary symmetric tensor categories
\cite{APS-contragredient}.

The present paper develops some of the general tools needed for this
program. Let $\C$ be a symmetric tensor category over an algebraically
closed field $\Bbbk$, and let $\g$ be an operadic Lie algebra in $\C$.
Although the natural map $\g\to U(\g)$ need not be injective
\cite{EtingofPBW}, its kernel is always central. If $\C$ is
quasi-semisimple, meaning that symmetric algebras preserve
monomorphisms, then the quotient by this kernel has the PBW property;
see \cref{prop:PBW-kernel-central}. In particular, every simple operadic
Lie algebra in a quasi-semisimple symmetric tensor category has the
PBW property.

Our first main result is a categorical version of Engel's theorem: if the underlying object of a finite-dimensional Lie algebra $\g$ is
semisimple and every simple subobject of $\g$ acts nilpotently on a
finite-dimensional representation $V$, then $V$ admits a finite
filtration strictly lowered by $\g$; see
\cref{thm:engels-representation-theorem}.

 We then study invariant bilinear
forms. We call an operadic Lie algebra \emph{quadratic} if it admits a
symmetric invariant nondegenerate form, and \emph{strongly semisimple} if its
Killing form is nondegenerate. We prove that a finite-dimensional quadratic
operadic Lie algebra with no non-zero abelian ideals is semisimple. Applied to the Killing form, this yields a
categorical Cartan criterion, see \cref{thm:cartan-criterion}. We also show that
every derivation of a strongly semisimple Lie algebra is inner.

We next establish a rigidity result for Lie algebra homomorphisms. In
classical theory, the vanishing of  the first Chevalley--Eilenberg
cohomology gives rigidity only up to conjugation. In $\C$, however, a non-zero object
$X$ may be \emph{invariantless}, meaning that
$\Hom_{\C}(\one,X)=0$. For an invariantless $\aLie$, we prove that the affine
scheme of Lie algebra homomorphisms $\g\to\aLie$ is finite and reduced when
the relevant first cohomology vanishes. In particular, this holds when $\g$ is
linearly reductive. This result has no nontrivial analogue in ordinary vector
spaces, where the only invariantless object is zero.

We then define a graded category $\Ocat$ and
Verma modules $\Delta(X,d)$. As in classical highest-weight theory, every Verma module has
a unique graded-simple quotient, and every graded-simple object of $\Ocat$
arises in this way, see \cref{thm:simple-graded,thm:classification-simple-O}. In particular, for $\g$ in
$\Verp$, we prove $U(\g)$ is noetherian and
$\Ocat$ is abelian. The key categorical
input is that the symmetric algebra of the nonclassical part of $\g$ is
finite-dimensional; this is \cref{thm:noetherian-enveloping}.

We also define categorical Shapovalov maps to study the reducibility of Verma modules. In characteristic zero, a nondegenerate grading makes every Shapovalov determinant generically nonzero, which yields generic graded simplicity of Verma modules; see Theorem~\ref{thm:generic-irreducibility-char0}. In characteristic $p$, central $p$-th powers obstruct the corresponding result for ordinary Verma modules. Under a $p$-nilpotence hypothesis in $\Verp$, we prove instead that baby Verma modules are graded-simple on a nonempty Zariski open set; see Theorem~\ref{thm:generic-irreducibility-verp}. We also construct an example illustrating the category $\mathcal O$ and Shapovalov maps.

The final part of the paper concerns contragredient Lie algebras in quasi-semisimple symmetric tensor categories. In this setting, the Cartan matrix is replaced by representation-theoretic
data \cite{APS-contragredient}. We prove
a simplicity criterion and show that, whenever the
relevant root strings are finite, reflection of the generating data is
realized by a Lie algebra isomorphism; see
\cref{thm:contragredient-indec-simple,thm:Ti-isomorphism}.
 This allows us to develop a root system, see 
\cref{th:cartan-graph-root-system}. We end this section with the computation of an explicit example.

\section{Lie algebras in symmetric tensor categories}\label{sec:conventions}

A \emph{pseudo-tensor category} over $\Bbbk$ is a $\Bbbk$-linear Karoubian rigid
monoidal category whose tensor product is $\Bbbk$-bilinear, with
finite-dimensional Hom spaces and
$\End(\one)=\Bbbk$.  A \emph{tensor category} is a pseudo-tensor category that
is abelian and in which every object has finite length.  When a pseudo-tensor category is endowed with a  braided structure, we say the category is braided.
If the braiding satisfies that the square of the braid is trivial for every pair of objects, we call the category symmetric. 
 For the general theory of tensor categories, we follow 
\cite{EGNO}.

Unless stated otherwise, $\C$ is a symmetric tensor category with
simple unit over the algebraically closed field $\Bbbk$; it is not
assumed semisimple. Tensor product and duality are exact. We suppress
associativity constraints and symmetry morphisms, and write
$\indcat{\C}$ for the ind-completion. An object of $\indcat{\C}$ is
called \emph{finite-dimensional} if it belongs to $\C$.

We call $\C$ \emph{quasi-semisimple} if every monomorphism $X\to Y$
induces a monomorphism $S(X)\to S(Y)$
\cite[Definition~7.5]{EtingofPBW}. Every semisimple symmetric tensor
category is quasi-semisimple. In characteristic zero, every symmetric
tensor category is quasi-semisimple, since symmetric powers are the
images of symmetrizer idempotents on tensor powers. In positive characteristic quasi-semisimplicity may fail, although all
known counterexamples are in characteristic $2$
\cite[\S7.3]{EtingofPBW}.

\subsection{Schematic notation}

We frequently express identities of morphisms using element notation.  If
$X\in\C$, a symbol $x\in X$ appearing in such a formula is a placeholder
for a tensor factor; it is not assumed to be a classical element
$\one\to X$.  Thus this notation remains meaningful even when
$\Hom_{\C}(\one,X)=0$.

More precisely, let $A$ be a commutative algebra object of $\C$.  The set
of $A$-valued points of an object $X$ is
\[
 X(A):=\Hom_{\C}(\one,A\otimes X)
 \cong
 \Hom_{A\text{-}\mathrm{mod}}(A,A\otimes X).
\]
A morphism $f:X_1\otimes\cdots\otimes X_n\longrightarrow Y$
induces a map
 $ f_A:X_1(A)\times\cdots\times X_n(A)\longrightarrow Y(A).$ 
It is obtained by tensoring the given $A$-valued points, using the
symmetry to move the copies of $A$ next to one another, multiplying them
in $A$, and then applying $f$. 

At the categorical level, every such formula denotes an equality between
morphisms with the same source and target, with the necessary
associativity and symmetry morphisms inserted.  For example, if
$\ttb:\g\otimes\g\to\g$ is a Lie bracket, then
\[
 [x,y]=-[y,x]
\]
denotes the morphism identity
\[
 \ttb=-\ttb\circ c_{\g,\g}.
\]
For further examples of this notation, see
Definition~\ref{def:orthogonal} and the proof of
Lemma~\ref{lemma:perp}.

\subsection{Basic notions}
For $V\in\C$, write $\wedge^nV$ for the ordinary exterior power, i.e., the
object of alternating coinvariants, and
\[
 \Lambda^nV:=(\wedge^nV^*)^*
\]
for the divided exterior power, i.e., the object of alternating invariants.
There is a natural morphism $\wedge^nV\to\Lambda^nV$.  It is an isomorphism
when $n!$ is invertible.  It is also an isomorphism for every
$V\in\Verplus$ and every $n$; see
\cite[Lemma~2.2]{EtingofHu}.

For $V\in\C$, a \emph{classical element} of $V$ is a morphism
$x:\one\to V$.  We set
\[
 \mathfrak{gl}(V):=\underline{\End}(V):=V\otimes V^*
\]
with its associative multiplication and commutator bracket.  If
$A\subseteq\underline{\End}(V)$ is a subobject, then
$A^r\subseteq\underline{\End}(V)$ denotes the image of the $r$-fold
composition map.

\begin{definition}\label[definition]{def:lie-algebra}
An \emph{operadic Lie algebra} in a symmetric pseudo-tensor category is an
object $\g$ with a skew-symmetric bracket
$[\ ,\ ]:\g\otimes\g\to\g$ satisfying the Jacobi identity.  If the
category is a symmetric tensor category, its enveloping algebra is
schematically
\[
 U(\g):=T(\g)\big/\bigl(x\otimes y-y\otimes x-[x,y]\bigr).
\]
The tensor-degree filtration gives a natural surjection
\[
 S(\g)\longrightarrow\gr U(\g).
\]
We call $\g$ a \emph{Lie algebra} if this map is an isomorphism.  This is
the PBW property.
\end{definition}

\begin{definition}\label[definition]{def:simple-semisimple-perfect}
We say that an operadic Lie algebra $\g$ is 
\begin{itemize}
\item \emph{simple} if it is nonabelian and has no
non-zero proper ideals,
\item \emph{semisimple} if it is a finite direct sum
of simple ideals,
\item \emph{perfect} if $[\g,\g]=\g$.
\end{itemize}
\end{definition}

A representation of $\g$ on $V$ is an operadic Lie algebra morphism
$\g\to\mathfrak{gl}(V)$.  We write $V^{\g}$ for the largest subobject of $V$ annihilated by $\g$; equivalently,
\[
 V^{\g}:=\ker\bigl(V\longrightarrow\g^*\otimes V\bigr),
\]
where the arrow is adjoint to the action $\g\otimes V\to V$.  In particular, the \emph{adjoint representation}
\begin{align}\label{eq:adjoint-rep}
\ad:\g\longrightarrow\underline{\End}(\g)
\end{align} 
is a morphism of operadic Lie algebras by the Jacobi identity.

An \emph{ideal} of $\g$ is a subobject stable under
the adjoint action.

\smallskip

In characteristic zero the PBW property is automatic \cite{Humphreys, EtingofPBW}.  In positive
characteristic it is an additional condition.  For a finite-dimensional
operadic Lie algebra $\g\in\C$, let
\[
 I(\g):=\ker\bigl(\g\longrightarrow U(\g)\bigr).
\]

\begin{proposition}\label[proposition]{prop:PBW-kernel-central}
The subobject $I(\g)$ is a central ideal of $\g$. If $\C$ is
quasi-semisimple, then $\g/I(\g)$ is a Lie algebra. In particular,
every simple operadic Lie algebra in a quasi-semisimple symmetric
tensor category is a Lie algebra.
\end{proposition}

\begin{proof}
The image of the natural map $\g\to U(\g)$ lies in the commutator Lie
algebra of $U(\g)$, so its kernel is an ideal.  
The
adjoint representation \eqref{eq:adjoint-rep} extends uniquely to an algebra morphism
$U(\g)\to\underline{\End}(\g)$.  It follows that
$I(\g)\subseteq\ker(\ad)$.

Assume now that $\C$ is quasi-semisimple. The induced map
$\g/I(\g)\to U(\g)$ is a monomorphism of operadic Lie algebras.
By \cite[Proposition~7.7]{EtingofPBW}, its source has the PBW property,
so it is a Lie algebra in the sense of \cref{def:lie-algebra}.
If $\g$ is simple, then the central ideal $I(\g)$ is zero, since
$\g$ is nonabelian. The final assertion follows.
\end{proof}

\begin{lemma}\label{{lem:split-subalgebra-PBW}}
Let $j:L\hookrightarrow A$ be a monomorphism of operadic Lie algebras
in a symmetric tensor category. Suppose that $A$ has the PBW property
and that $j$ has a retraction as a morphism of underlying objects.
Then $L$ has the PBW property. In particular, this applies when $A$
is an associative algebra with its commutator bracket.
\end{lemma}
\begin{proof}
    Let $\eta_L:S(L)\twoheadrightarrow\gr U(L)$ and
$\eta_A:S(A)\twoheadrightarrow\gr U(A)$ be the canonical morphisms.
Naturality gives
\[
 \gr U(j)\circ\eta_L=\eta_A\circ S(j).
\]
The right-hand side is a monomorphism: $\eta_A$ is an isomorphism,
and $S(j)$ is split injective because the symmetric-algebra functor
preserves the chosen retraction. Thus $\eta_L$ is a monomorphism,
and hence an isomorphism. The last assertion uses
\cite[Corollary~7.3]{EtingofPBW}. 
\end{proof}

For $\C=\Verp$, the PBW property is equivalent to the $p$-Jacobi condition
of \cite[Theorem~6.6]{EtingofPBW}.  The preceding proposition gives, in
particular, that $I(\g)\subseteq\ker(\ad)=Z(\g)$; compare
\cite[Proposition~2.3]{Newton}.

An object $X\in\C$ is \emph{invariantless} if
$\Hom_{\C}(\one,X)=0$.  A finite-dimensional Lie algebra $\g$ is
\emph{linearly reductive} if the category of its finite-dimensional modules
in $\C$ is semisimple.

If $\C$ is quasi-semisimple, then
\cref{prop:PBW-kernel-central} implies that every simple or semisimple
operadic Lie algebra is a Lie algebra. A simple operadic Lie algebra is perfect and has
zero center.  Distinct simple ideals in a direct sum commute.

\section{Engel's theorem}\label{sec:engel}
In this section $\C$ is an arbitrary symmetric tensor category, and
$\Bbbk$ need not be algebraically closed.

Let $V\in \C$. A subobject 
$U\subseteq\underline{\End}(V)$
 is \emph{nilpotent} if $U^n=0$ for some $n\geq1$. We will need the following lemmas.

\begin{lemma}\label{lem:nilpotent-filtration}
A subobject $U\subseteq\underline{\End}(V)$ is nilpotent if and only if there is
a finite filtration
\[
0=F_0V\subseteq F_1V\subseteq\cdots\subseteq F_NV=V
\]
such that
\[
U\cdot F_iV\subseteq F_{i-1}V
\qquad\text{for every }1\leq i\leq N.
\]
\end{lemma}

\begin{proof}
Let $a_i:U^{\otimes i}\otimes V\to V$ be the iterated action and let $\widetilde a_i:V\to(U^{\otimes i})^*\otimes V$ be its adjoint. Set
$ F_iV:=\ker\widetilde a_i.$ 
Then $F_0V=0$, $F_NV=V$, and
$U\cdot F_iV\subseteq F_{i-1}V$.
Conversely, the inclusions imply that 
$U^N\cdot V=0$, and hence $U^N=0$.
\end{proof}

\begin{lemma}\label[lemma]{lem:ad-nilpotent}
If $U\subseteq\underline{\End}(V)$ satisfies $U^N=0$ for some $N\ge1$, then the commutator action of $U$ on
$\underline{\End}(V)$ satisfies
\[
(\ad U)^{2N-1}=0.
\]
\end{lemma}

\begin{proof}
Put $m=2N-1$. We expand
\[
 (\ad u_1)\cdots(\ad u_m)(T),
 \qquad u_i\in U,\quad T\in\underline{\End}(V).
\]
Each summand is a product of $a$ factors from $U$, followed by $T$, followed by $m-a$ factors from $U$. Since either $a\ge N$ or $m-a\ge N$, every summand vanishes by $U^N=0$. Thus the entire morphism $U^{\otimes m}\otimes\underline{\End}(V)\to\underline{\End}(V)$ is zero.
\end{proof}

Let $\g$ be a finite-dimensional Lie algebra in $\C$. For a Lie subalgebra $\h\subseteq\g$, we consider the composite
\[
\g\otimes\h\xrightarrow{[\, ,\, ]}\g\longrightarrow\g/\h.
\]
By adjunction, this determines a morphism
\[
\g\longrightarrow(\g/\h)\otimes\h^*.
\]
The \emph{normalizer} $N_\g(\h)$ is defined to be its kernel.
Schematically,
$N_\g(\h)=\{x\in\g:[x,\h]\subseteq\h\}.$
The Jacobi identity implies that $N_\g(\h)$ is a Lie subalgebra of
$\g$.

\begin{theorem}[Engel's theorem for representations]
\label{thm:engels-representation-theorem}
Let $V\in\C$, and let $\g$ be a finite-dimensional Lie algebra
acting on $V$ whose underlying object in $\C$ is semisimple. Assume
that every simple subobject $X\subseteq\g$ acts nilpotently on $V$. Then there is a finite filtration
\[
0=F_0V\subseteq F_1V\subseteq\cdots\subseteq F_NV=V
\]
such that
\[
\g\cdot F_iV\subseteq F_{i-1}V
\qquad\text{for every }1\leq i\leq N.
\]
\end{theorem}

\begin{proof}

Every subquotient of the underlying object of $\g$ is semisimple.
In particular, every simple subobject of its image lifts to a simple
subobject of $\g$, so the hypothesis is preserved. The image still has
the PBW property, as a quotient of $\g$
\cite[Lemma~7.6(i)]{EtingofPBW}. We may therefore assume that
$\g\subseteq\underline{\End}(V)$.

We first prove that $V^{\g}\ne0$ when $V\ne0$.  We argue by induction on
the length of $\g$. All underlying subquotients of $\g$ are semisimple. The Lie
subalgebras used in the induction have the PBW property by Lemma \ref{{lem:split-subalgebra-PBW}}, and their images have PBW by
\cite[Lemma~7.6(i)]{EtingofPBW}. Their underlying objects are
semisimple, since they are subquotients of $\g$. The statement is clear for $\g=0$.

Let  $\h\subsetneq\g$ be a proper Lie subalgebra of maximal length.  
Note that, if $\h=0$,  then for any simple subobject $X\subseteq\g$,  the Lie subalgebra
generated by $X$ is non-zero, hence equals $\g$.  Since $X$ is nilpotent in
$\underline{\End}(V)$, the nonunital associative algebra generated by $X$ is
nilpotent.  It contains $\g$, so $V^{\g}\ne0$ by
Lemma \ref{lem:nilpotent-filtration}, as desired. 

Assume then that $\h\ne 0.$
Every
simple subobject of $\h$ acts nilpotently on $V$.  By
\cref{lem:ad-nilpotent}, it also acts nilpotently on $\g/\h$ through the
adjoint action.  By the induction hypothesis applied to the image of $\h$ in
$\underline{\End}(\g/\h)$, there exists a non-zero simple subobject
\[
 \overline X\subseteq(\g/\h)^{\h}.
\]
Since the underlying object of $\g$ is semisimple, the quotient map
$\g\twoheadrightarrow\g/\h$ splits in $\C$; hence
$\overline X$ lifts to a simple subobject $X\subseteq\g$.  Then $X\not\subseteq\h$ and $[\h,X]\subseteq\h$.
Thus $N_{\g}(\h)$ strictly contains $\h$.  Since $\h$ is of maximal length, we get that 
$N_{\g}(\h)=\g$, so $\h$ is an ideal.

By induction, $W:=V^{\h}$ is non-zero.  Since
$\h$ is an ideal, $W$ is stable under $\g$, and $\g/\h$ acts on $W$.
Every simple subobject of the image of $\g/\h$ in $\underline{\End}(W)$ lifts to a
simple subobject of $\g$ and is nilpotent.  Applying the induction
hypothesis once more we obtain that
\[
 0\ne W^{\g/\h}=V^{\g}.
\]

We now construct the filtration.  Choose a simple subobject
$Y\subseteq V^{\g}$.  The hypothesis descends to the induced action on $V/Y$.  By induction on the
length, we have a filtration for $V/Y$ strictly lowered by $\g$. Hence we get a
filtration of $V$ strictly lowered by $\g$ by pullback of this filtration of $V/Y$.
\end{proof}

The \emph{lower central series} of $\g$ is the descending sequence of
ideals
\[
\g=\gamma_1(\g)\supseteq\gamma_2(\g)\supseteq
\gamma_3(\g)\supseteq\cdots
\]
defined recursively by
\[
\gamma_{r+1}(\g):=[\g,\gamma_r(\g)].
\]
We call $\g$ \emph{nilpotent} if $\gamma_N(\g)=0$ for some
$N$.

\begin{corollary}[Engel's theorem]\label[corollary]{cor:engel}
Let $\g$ be a finite-dimensional Lie algebra whose underlying object
in $\C$ is semisimple. Then $\g$ is nilpotent if and only if every
simple subobject $X\subseteq\g$ acts nilpotently on $\g$ through the
adjoint action.
\end{corollary}

\begin{proof}
The forward implication is straightforward.  For the
converse, we apply \cref{thm:engels-representation-theorem} to the image of
$\ad:\g\to\mathfrak{gl}(\g)$.  The resulting filtration of $\g$ is
strictly lowered by the adjoint action, so the lower central series
terminates.
\end{proof}

\begin{remark}
Our proof closely follows the proof of Engel's theorem for generalized
Lie algebras given in \cite{KochetovRadu}.
\end{remark}

\begin{remark}
    The semisimplicity hypothesis on the underlying object of $\g$ cannot
be dropped. In characteristic zero, take
$\C=\Rep(\mathbb G_a)$ and let
\[
 \g=\Bbbk h\oplus\Bbbk e,\qquad [h,e]=e,\qquad
 t\cdot h=h+te,\quad t\cdot e=e.
\]
The displayed $\mathbb G_a$-action preserves the bracket. The only
simple subobject of $\g$ is $\Bbbk e$, and its adjoint action is
nilpotent, since $(\ad e)^2=0$. Nevertheless,
$(\ad h)(e)=e$, so the adjoint action of $\g$ is not nilpotent;
indeed, $\gamma_n(\g)=\Bbbk e$ for every $n\ge2$. 
\end{remark}

\section{Quadratic Lie algebras, the Killing form, and derivations}
\label{sec:quadratic}

Throughout this section, $\C$ is an arbitrary symmetric tensor
category and $\operatorname{char}\Bbbk\ne2$.

\subsection{Invariant bilinear forms}
Let $V$ be a $\g$-module, with action
$\rho:\g\otimes V\to V$. A bilinear form
$B:V\otimes V\to\one$ is \emph{invariant} if
\[
B\circ(\rho\otimes\id_V)
+
B\circ(\id_V\otimes\rho)
\circ(c_{\g,V}\otimes\id_V)
=0.
\]
Schematically, this means
\begin{equation}\label{eq:invariant-form}
B(a\cdot v,w)+B(v,a\cdot w)=0.
\end{equation}
Equivalently, the induced morphism $V\to V^*$ is a morphism of
$\g$-modules. The left and right radicals of an invariant bilinear form are
$\g$-submodules.

For the adjoint representation $V=\g$, invariance can be
written (schematically) as
\[
B([a,b],c)=B(a,[b,c]).
\]

\begin{definition}\label[definition]{def:quadratic}
An operadic Lie algebra $\g$ is \emph{quadratic} if it admits a symmetric
invariant nondegenerate bilinear form.
\end{definition}

\begin{definition}
\label{def:orthogonal}
Let $\phi:\mathfrak g \otimes \mathfrak g \to \mathbf{1}$ be a symmetric bilinear form, and let $\mathfrak h \subset \mathfrak g$ be a subobject. The orthogonal complement $\mathfrak h^\perp$ is defined as the kernel of the composite
\[ p_{\h}:
\mathfrak g \xrightarrow{\tilde \phi} \mathfrak g^* \twoheadrightarrow \mathfrak h^*,
\]
where $\tilde \phi:\mathfrak g \to \mathfrak g^*$ is the map induced from $\phi$ by the natural adjunction. Schematically, $\h^\perp$ is the set of $x\in \g$ such that 
\[
\phi(x,y)=0,\quad \text{for all }y\in\h.
\]

\end{definition}

\begin{lemma}\label{lemma:perp}
Let $\phi$ be a symmetric invariant bilinear form. If $\mathfrak h \subset \mathfrak g$ is an ideal, then $\mathfrak h^\perp$ is also an ideal.
\end{lemma}

\begin{proof} We give both the morphism-theoretic argument and its schematic
translation, to further illustrate our use of schematic notation.
It suffices to show that the composite
\[
\mathfrak g\otimes \mathfrak h^\perp
\xrightarrow{[\ ,\ ]}
\mathfrak g
\xrightarrow{p_{\h}}
\mathfrak h^*
\]
is zero. Equivalently, using the adjunction, it suffices to show that the associated morphism
\[
\phi\circ(\mathrm{id}_{\mathfrak h}\otimes[\ ,\ ]):\mathfrak h\otimes \mathfrak g\otimes \mathfrak h^\perp \to \mathbf 1
\]
is zero.
By invariance of \(\phi\), this is the same as the morphism $\phi\circ([\ ,\ ]\otimes \mathrm{id}_{\mathfrak h^\perp})$, which is zero. In fact, 
because $\mathfrak h$ is an ideal, the bracket $[\mathfrak h \otimes \mathfrak g]\to \mathfrak g$
factors through $\mathfrak h$. Following this by $\phi$ thus results in zero, as the pairing of $\mathfrak h^\perp$ with
$\mathfrak h$ is zero.

Schematically, if
$a\in\g$, $x\in\h^\perp$, and $y\in\h$, then invariance gives
\[
\phi([a,x],y)=-\phi(x,[a,y])=0,
\]
because $[a,y]\in\h$. Hence $[a,x]\in\h^\perp$, so
$\h^\perp$ is an ideal.
\end{proof}

\begin{proposition}\label[proposition]{prop:forms-on-simple}
Let $\g$ be simple.  Every non-zero invariant bilinear form on $\g$ is
symmetric and nondegenerate.  If such a form exists, it is unique up to a
scalar.  In particular, $\g$ is quadratic.
\end{proposition}

\begin{proof}
The left and right radicals are ideals.  Since the form is non-zero and $\g$
is simple, both radicals are zero.  Thus $\g\cong\g^*$ as a $\g$-module,
and Schur's lemma shows that the space of invariant forms is
one-dimensional.

Let $B^t(x,y):=B(y,x)$.  Since $B^t$ is invariant, $B^t=cB$ for some
$c\in \Bbbk^\times$.  Transposing twice gives $c^2=1$, so $c=1$ or $c=-1$.
Suppose $c=-1$.  Set $T(x,y,z):=B([x,y],z)$.  Invariance and skew-symmetry
give
\[
T(x,y,z)=-T(y,z,x).
\]
Applying this identity three times gives $T=-T$.  Hence $T=0$ because
$\operatorname{char}\Bbbk\ne2$.  Since $\g$ is perfect, this implies $B=0$, a
contradiction.  Therefore $c=1$.
\end{proof}

The following theorem relates quadratic operadic Lie algebras with semisimplicity. We will use it in the following subsection to prove Cartan's criterion. 
\begin{theorem}\label{thm:quadratic-semisimple}
Let $(\g,B)$ be a finite-dimensional quadratic operadic Lie algebra.  If
$\g$ has no non-zero abelian ideals, then $\g$ is semisimple.
\end{theorem}

\begin{proof}
We use induction on the length of $\g$.  Let $\mathfrak m\subseteq\g$ be a
minimal non-zero ideal.  The ideal $[\mathfrak m,\mathfrak m]$ is either zero
or $\mathfrak m$.  The first case is excluded by the hypothesis, so
$[\mathfrak m,\mathfrak m]=\mathfrak m$.

The intersection $\mathfrak m\cap\mathfrak m^\perp$ is an ideal contained
in $\mathfrak m$.  It is therefore zero or $\mathfrak m$.  If
$\mathfrak m\subseteq\mathfrak m^\perp$, then by invariance we obtain 
\[
 B([x,y],z)=B(x,[y,z])=0,
 \qquad x,y\in\mathfrak m,\quad z\in\g.
\]
Since $[\mathfrak m,\mathfrak m]=\mathfrak m$, this implies
$B(\mathfrak m,\g)=0$, contradicting nondegeneracy.  Hence
$\mathfrak m\cap\mathfrak m^\perp=0$.

The map $\g\to\mathfrak m^*$ induced by $B$ has kernel
$\mathfrak m^\perp$.  Its restriction $\mathfrak m\to\mathfrak m^*$ is a
non-zero morphism between simple $\g$-modules, hence an isomorphism. Hence, if $i:\mathfrak m\hookrightarrow\g$ is the
inclusion, $b:\g\to\g^*$ is induced by $B$, and
$b_{\mathfrak m}=i^*bi$, then
\[
 r=b_{\mathfrak m}^{-1}i^*b:\g\longrightarrow\mathfrak m
\]
is a retraction of $i$, with kernel $\mathfrak m^\perp$. Thus
\[
 \g=\mathfrak m\oplus\mathfrak m^\perp
\]
as a direct sum of ideals.  Since the two ideals commute, every ideal of
$\mathfrak m$ is an ideal of $\g$. By the minimality assumption, $\mathfrak m$ is
simple.  Any abelian ideal of $\mathfrak m^\perp$ is an abelian ideal of
$\g$, so $\mathfrak m^\perp$ satisfies the same hypotheses.  The proof follows by induction.
\end{proof}

\subsection{The Killing form and Cartan's criterion}

\begin{definition}
The \emph{Killing form} of a Lie algebra $\mathfrak g$ is the morphism
\[
K_{\g} : \mathfrak g \otimes \mathfrak g \to \mathbf{1}
\]
defined as the composition
\[
\mathfrak g \otimes \mathfrak g
\xrightarrow{\ \operatorname{ad} \otimes \operatorname{ad}\ }
(\mathfrak g \otimes \mathfrak g^*) \otimes (\mathfrak g \otimes \mathfrak g^*)
\xrightarrow{\operatorname{ev}} \mathfrak g \otimes \mathfrak g^*
\xrightarrow{ \tr_{\g}=\operatorname{ev} \ }
\mathbf{1}.
\]
Schematically,
\begin{equation}\label{eq:killing-form}
K_{\g}(x,y):=\tr_{\g}(\ad(x)\ad(y)).
\end{equation}
\end{definition}
Notice that the Killing form  is symmetric and invariant.

\begin{definition}\label[definition]{def:strongly-semisimple}
An operadic Lie algebra $\g$ in a symmetric pseudo-tensor category is
\emph{strongly semisimple} if its Killing form induces an isomorphism
$\g\xrightarrow{\sim}\g^*$.
\end{definition}

Thus, a strongly semisimple operadic Lie algebra is quadratic, with its
Killing form as a distinguished quadratic form.
We will need the following technical lemmas.

\begin{lemma}\label[lemma]{lem:trace-nilpotent-subobject}
If $U\subseteq\underline{\End}(V)$ is nilpotent, then the trace morphism
$U\to\one$ is zero.
\end{lemma}

\begin{proof}
Let
\(
a:U\otimes V\longrightarrow V
\)
be the evaluation action. We argue by induction on its nilpotence
degree. If $a=0$, the result is immediate.
Otherwise, let $W=\ima(a)$ and write
\[
a:U\otimes V\xrightarrow{b}W\xrightarrow{i}V.
\]
The action of $U$ on $W$ is given by
\[
a_W=b(\id_U\otimes i):U\otimes W\longrightarrow W.
\]
If the $N$-fold action of $U$ on $V$ is zero, then the
$(N-1)$-fold action on $W$ is zero. Indeed, this follows from the
factorization above, since $i$ is a monomorphism and tensor product
preserves epimorphisms. Thus, by induction, $\tr_W(a_W)=0$.
Finally, 
\[
\tr_V(a)
=\tr_V(ib)
=\tr_W\bigl(b(\id_U\otimes i)\bigr)
=\tr_W(a_W)
=0.
\]
Since $\tr_V(a):U\to\one$ is the restriction to $U$ of the
categorical trace,
the result follows.
\end{proof}

\begin{lemma}\label[lemma]{lem:killing-no-abelian-ideal}
A strongly semisimple operadic Lie algebra has no non-zero abelian ideals.
\end{lemma}

\begin{proof}
Let $\aLie\subseteq\g$ be an abelian ideal and set
\[
 A:=(\ad\aLie)(\ad\g)\subseteq\underline{\End}(\g).
\]
Then $A^2=0$.  Indeed,
\[
 [\aLie,[\g,[\aLie,[\g,\g]]]]
 \subseteq[\aLie,[\g,\aLie]]
 \subseteq[\aLie,\aLie]=0.
\]
By \cref{lem:trace-nilpotent-subobject}, the trace vanishes on $A$.
Therefore $K_{\g}(\aLie,\g)=0$.  By strong semisimplicity, we get that $\aLie=0$.
\end{proof}

The following proof of Cartan's criterion is the categorical version of the
short classical argument of Dieudonn\'e \cite{Dieudonne}.

\begin{theorem}[Cartan's criterion]\label{thm:cartan-criterion}
Every finite-dimensional strongly semisimple operadic Lie algebra in $\C$ is
semisimple.  In particular, it is a Lie algebra.
\end{theorem}

\begin{proof}
We apply \cref{lem:killing-no-abelian-ideal,thm:quadratic-semisimple}
to the Killing form to obtain that $\g$ is semisimple.

To prove the PBW property, set $A=\underline{\End}(\g)$ and let
$j=\ad:\g\to A$. We define
\[
 q:A\longrightarrow\g^*,\qquad
 q(T)(x)=\tr_{\g}(T\ad x).
\]
Then $qj=K_{\g}$, so $K_{\g}^{-1}q$ is a retraction of $j$ as
a morphism of underlying objects. The commutator Lie algebra of $A$
has the PBW property by \cite[Corollary~7.3]{EtingofPBW}. Consequently the canonical
map $\eta_{\g}:S(\g)\twoheadrightarrow\gr U(\g)$ is injective:
indeed,
\[
 \gr U(j)\circ\eta_{\g}=\eta_A\circ S(j),
\]
and the right-hand side is a monomorphism because $S(j)$ is split
injective. Thus $\eta_{\g}$ is an isomorphism, as required.
\end{proof}

If $\operatorname{char}\Bbbk=0$, every finite-dimensional semisimple Lie
algebra in $\operatorname{Vec}_{\Bbbk}$ is strongly semisimple. In positive characteristic, the implication
from semisimplicity to strong semisimplicity can fail. Likewise, even in characteristic zero, a
semisimple Lie algebra in $\operatorname{sVec}_{\Bbbk}$ may have degenerate
Killing form; see, for example, \cite{Kac,Musson}.

\begin{proposition}\label[proposition]{prop:strong-semisimplicity-preserved}
Let $F:\mathcal D\to\mathcal E$ be an additive symmetric monoidal functor
between symmetric pseudo-tensor categories.  If an operadic Lie algebra
$\g\in\mathcal D$ is strongly semisimple, then $F(\g)$ is strongly
semisimple.
\end{proposition}

\begin{proof}
A symmetric monoidal functor preserves duals, categorical traces, and the
Killing form.  It sends the isomorphism $K_{\g}:\g\to\g^*$ to the
isomorphism
$K_{F(\g)}:F(\g)\longrightarrow F(\g)^*$,
and the statement follows.
\end{proof}

\subsection{Derivations} In this subsection we study derivations; the main argument here is formally identical to the classical
proof that every derivation of a finite-dimensional semisimple Lie
algebra in characteristic zero is inner; see
\cite[Section~5.3]{Humphreys}.

Let $\Der(\g)\subseteq\underline{\End}(\g)$ be the kernel of the Leibniz
defect
\[
 d\longmapsto
 d\circ[\ ,\ ]-[\ ,\ ]\circ(d\otimes\id+\id\otimes d).
\]
It is an operadic Lie subalgebra of $\underline{\End}(\g)$, and the
adjoint map $\ad:\g\to\Der(\g)$ is a morphism of operadic Lie
algebras.

\begin{theorem}\label{thm:all-derivations-inner}
If $\g$ is strongly semisimple, then
\(
 \ad:\g\xrightarrow{\sim}\Der(\g)
\)
is an isomorphism.
\end{theorem}

\begin{proof}
The center of $\g$ lies in the radical of the Killing form, so it is zero.
Thus $\ad$ is injective.  On $\Der(\g)$ consider the trace form
\(
 T(d,e):=\tr_{\g}(de).
\)
Its restriction to $\ad(\g)$ is the Killing form.  Hence this restriction
is nondegenerate, and
\[
 \Der(\g)=\ad(\g)\oplus I,
 \qquad
 I:=\ad(\g)^\perp.
\]
The trace form is invariant under the commutator bracket.  Since
$\ad(\g)$ is an ideal of $\Der(\g)$, its orthogonal complement $I$ is also
an ideal, see Lemma \ref{lemma:perp}.  Therefore
\[
 [I,\ad(\g)]\subseteq I\cap\ad(\g)=0.
\]
For $d\in I$ and $x\in\g$, the derivation identity gives
\[
 [d,\ad(x)]=\ad(d(x)).
\]
Hence $d(x)$ is central for every $x$.  The center is zero, so $d=0$.
Thus $I=0$.
\end{proof}

\section{Rigidity of Lie algebra homomorphisms}\label{sec:homomorphism-rigidity}

In this section, let $\C$ be an arbitrary symmetric tensor category. Let $H^1_{\CE}(\g,M)$ denote the ordinary vector space of
$1$-cocycles $c:\g\to M$ modulo the coboundaries coming from classical
elements $m:\one\to M$.  Equivalently, it is the first cohomology of the
complex of morphism spaces obtained from the internal Chevalley--Eilenberg
complex \cite{EtingofHu} by applying $\Hom_{\C}(\one,-)$.

For finite-dimensional Lie algebras $\g$ and $\aLie$ in $\C$, the Lie algebra
homomorphisms $\g\to\aLie$ form a closed affine subscheme
\[
 \underline{\operatorname{Hom}}_{\mathrm{Lie}}(\g,\aLie)
 \subseteq \Hom_{\C}(\g,\aLie).
\]
The defining equations say that a linear morphism preserves the bracket.  It is clear that this affine scheme is of finite type, and we will use this fact in what follows. 

\begin{proposition}\label[proposition]{prop:finite-Hom-scheme}
Let $\g$ and $\aLie$ be finite-dimensional Lie algebras in $\C$.  Assume that
$\aLie$ is invariantless and that
\[
 H^1_{\CE}(\g,M)=0
\]
for every finite-dimensional invariantless $\g$-module $M$.  Then
$\underline{\operatorname{Hom}}_{\mathrm{Lie}}(\g,\aLie)$ is a finite reduced
scheme.  In particular, the set of Lie algebra homomorphisms $\g\to\aLie$ is
finite.
\end{proposition}

\begin{proof}
Fix a homomorphism $f:\g\to\aLie$, and let $\aLie_f$ denote $\aLie$ with the
$\g$-action $\ad\circ f$.  A first-order deformation $f+\epsilon c$ over
$\Bbbk[\epsilon]/(\epsilon^2)$ preserves the bracket exactly when, schematically,
\[
 c([x,y])=x\cdot c(y)-y\cdot c(x).
\]
Thus the Zariski tangent space to the scheme at $f$ is
\[
 T_f\underline{\operatorname{Hom}}_{\mathrm{Lie}}(\g,\aLie)
 =Z^1_{\CE}(\g,\aLie_f).
\]
 Moreover,
\[
 C^0_{\CE}(\g,\aLie_f)=\Hom_{\C}(\one,\aLie)=0,
\]
so $B^1_{\CE}(\g,\aLie_f)=0$.  The hypothesis therefore gives
$Z^1_{\CE}(\g,\aLie_f)=0$.

Hence every Zariski tangent space of the finite-type affine scheme
$\underline{\operatorname{Hom}}_{\mathrm{Lie}}(\g,\aLie)$ is zero.  
Thus every closed point is reduced and isolated.  A finite-type affine scheme
over an algebraically closed field with this property is finite and reduced.
\end{proof}

\begin{corollary}\label[corollary]{cor:finite-Hom-linearly-reductive}
If $\g$ is linearly reductive and $\aLie$ is invariantless, then the set
$\operatorname{Hom}_{\mathrm{Lie}}(\g,\aLie)$ is finite.
\end{corollary}

\begin{proof}
Let $M$ be a finite-dimensional $\g$-module.  A $1$-cocycle
$c:\g\to M$ defines an extension of $\g$-modules
\[
 0\longrightarrow M\longrightarrow M\oplus\one
 \longrightarrow\one\longrightarrow0
\]
by the schematic formula
\[
 x\cdot(m,t)=x\cdot m+t c(x).
\]
By the linear reductivity assumption this extension splits, and so $c$ is a coboundary.  Hence $H^1_{\CE}(\g,M)=0$ for every $M$, and we can apply 
\cref{prop:finite-Hom-scheme}.
\end{proof}
\begin{remark}
    With our definition, the existence of a linearly reductive Lie
algebra in $\C$ forces $\C$ itself to be semisimple: every short
exact sequence in $\C$ can be equipped with trivial Lie algebra
actions. 
\end{remark}

\begin{remark}
The tangent space $Z^1_{\CE}(\g,\aLie_f)$ describes infinitesimal deformations
of $f$.  The subspace $B^1_{\CE}(\g,\aLie_f)$ consists of infinitesimal
conjugations by classical elements of $\aLie$.  Thus $H^1=0$ means that there
are no infinitesimal deformations modulo conjugation, while invariantlessness
of $\aLie$ removes the infinitesimal conjugations themselves.
\end{remark}

\section{Graded Lie algebras and category \texorpdfstring{$\mathcal O$}{O}}
\label{sec:graded}

Let $\C$ be an arbitrary simmetric tensor category. 
Next we develop several aspects of the representation theory of $\bZ$-graded Lie algebras. After some basic results we introduce the graded category $\cO$ and the Shapovalov determinant. We also include  examples where we explicitly compute the data involved here.

\subsection{$\bZ$-graded Lie algebras, Verma and simple modules}\label{subs: graded Lie-alg}

 Let $\g=\bigoplus_{i\in\mathbb Z}\g_i$ be a finite-dimensional $\mathbb Z$-graded Lie algebra.  Set
\begin{align*}
\h &:=\g_0, & \nplus&:=\bigoplus_{i>0}\g_i, & \nminus&:=\bigoplus_{i<0}\g_i, & \bLie^{\pm}&:=\h\oplus\mathfrak{n}^{\pm}.
\end{align*}
Thus $\g=\nminus\oplus\h\oplus\nplus$ as an object. The subalgebras $\nplus$ and $\nminus$ are nilpotent. The inclusions of $\nminus$, $\h$, $\nplus$, and $\bLie^\pm$
into $\g$ split as morphisms of underlying objects. By Lemma \ref{{lem:split-subalgebra-PBW}}, these subalgebras inherit the PBW property. Hence
\[
 U(\nminus)\otimes U(\h)\otimes U(\nplus)
 \xrightarrow{\sim}U(\g).
\]
Indeed, the associated graded map is the isomorphism induced by
\[
 S(\g)=S(\nminus)\otimes S(\h)\otimes S(\nplus).
\]

\smallbreak

Given a $\fg$-module $V$, we set $V^{\pm}$ as the maximal subobject of $V$ where $\fn_{\pm}$ acts trivially. Notice that $V^{\pm}$ is a $\fg_0$-submodule of $V$. 

\begin{lemma}\label{lem:submod-gen-from-V+}
Let $Y$ be a $\fh$-submodule of $V^+$. Then $W=U(\fn_-)\cdot Y$ is a $\fg$-submodule of $V$.
\end{lemma}
\begin{proof}
Clearly, $\fn_-\cdot W\subseteq W$, and also $\fh\cdot W\subseteq W$ since 
$\fh U(\fn_-)\subseteq U(\fn_-)+U(\fn_-)\fh$ and $Y$ is an $\fh$-submodule. Finally, as 
$\fn_+U(\fn_-) \subseteq U(\fb_-)+U(\fn_-)\fn_+$ and $\fn_+$ annihilates $Y$, we get that $\fn_+\cdot W\subseteq W$
\end{proof}

\smallbreak

Next we study $\bZ$-graded $\fg$-modules $M=\oplus_{k\in\bZ} M_k$. We define the shifted module $M\langle d\rangle$ by
\[
(M\langle d\rangle)_r:=M_{r-d}.
\]
Let $X$ be an irreducible $\fh$-module.  We regard $X$ as a $\bLie^+$-module by letting $\nplus$ act trivially.

\begin{definition}\label[definition]{def:verma-module}
For $d\in\mathbb Z$, the graded Verma module of highest weight $(X,d)$ is
\begin{equation}\label{eq:verma-definition}
\Delta(X,d):=U(\g)\otimes_{U(\bLie^+)}X\langle d\rangle.
\end{equation}
\end{definition}

\begin{theorem}\label{thm:simple-graded}
Let $X$ be an irreducible $\fh$-module, $d\in\mathbb Z$. The Verma module $\Delta(X,d)$ has a unique maximal proper graded submodule. 
Its graded-simple quotient is denoted by $L(X,d)$.
\end{theorem}

\begin{proof}
The component of degree $d$ of $\Delta(X,d)$ is the simple $\h$-module $X$, and it generates the Verma module.  A graded submodule is proper exactly when its component of degree $d$ is zero.  The sum of all proper graded submodules is therefore still proper and is the unique maximal proper submodule.
\end{proof}

\begin{lemma}\label{lem:L(S)+-is-S} With the notation of the theorem above, 
$L(X,d)^+=X$.
\end{lemma}
\begin{proof}
Notice that $L(X,d)^+$ is a $\bZ$-graded $\fh$-submodule of $L(X,d)$: 
\begin{align*}
L(X,d)^+ &=\oplus_{k\ge0} L(X,d)^+_{d-k}, & L(X,d)^+_j&\subseteq L(X,d)_j, & L(X,d)^+_d&=X.
\end{align*}
For each $k>0$, $L(X,d)^+_{d-k}$ is an $\fh$-submodule of $L(X,d)^+$, so by Lemma \ref{lem:submod-gen-from-V+} $W=U(\fn_-)\cdot L(X,d)^+_{d-k}$ is a $\fg$-submodule of $L(X,d)$. In addition, $W$ is graded. As $W\subsetneq L(X,d)$, we have that $W=0$, and in particular $W_{d-k}=L(X,d)^+_{d-k}=0$ for $k>0$.
\end{proof}

Next we give a generalization of \cite[Theorem 4]{KacW-irred} (which in turn was generalized in \cite{AA-super} for Lie superalgebras) for Lie algebras in symmetric tensor categories. We characterize simple $\fg$-modules in $\cC$ (instead of $\indcat{\C}$) under mild hypotheses. We write $L(X)$ for the underlying ungraded module of $L(X,0)$, and similarly forget the grading of $\Delta(X,0)$ when discussing ungraded modules. We return to the notation $\Delta(X,d)$ and $L(X,d)$ in Subsection~\ref{subsec:cat-O}.

\begin{theorem}\label{thm:irred-graded-modules}
Let $\fg=\oplus_{k\in\bZ}\fg_k \in\indcat{\cC}$ be a $\bZ$-graded Lie algebra generated by
\begin{align}\label{eq:gloc-defn}
\fg_{\operatorname{loc}}\coloneqq \fg_{-1}\oplus\fh\oplus\fg_{1}.
\end{align}
Let $V\in\cC$ be a $\fg$-module. Then the following are equivalent:
\begin{enumerate}[leftmargin=*,label=\rm{(\roman*)}]
\item\label{item:irred-graded-modules-L(S)} There exists a simple $\fh$-module $X$ such that $V\simeq L(X)$.
\item\label{item:irred-graded-modules-nilp-action} $V$ is a simple $\fg$-module and the action of $\fn_{\pm}$ is nilpotent.
\end{enumerate}
Moreover, two simple $\fg$-modules $V,W\in\cC$ with nilpotent action of $\fn_{\pm}$ are isomorphic if and only if $V^+\simeq W^+$ as $\fh$-modules.
\end{theorem}
\begin{proof}
\ref{item:irred-graded-modules-L(S)} $\implies$ \ref{item:irred-graded-modules-nilp-action} Let $X$ be a simple $\fh$-module such that $L(X)\in\cC$. As $L(X)$ is $\bZ$-graded with a finite number of non-zero homogeneous components, the action of both $\fn_+$ and $\fn_-$ is nilpotent. 

It remains to prove that $L(X)$ is simple. Let $W\ne 0$ be a $\fg$-submodule of $L(X)$. Choose the largest
$j\ge0$ for which the iterated image $(\nplus)^jW$ is nonzero.
This image is annihilated by $\nplus$, so $W^+\ne0$.
By \cref{lem:L(S)+-is-S},
$W^+\subseteq L(X)^+=X$ as $\h$-modules. Simplicity of $X$ gives
$W^+=X$, and hence
$L(X)=U(\g)X\subseteq W$. Thus $W=L(X)$.

\noindent\ref{item:irred-graded-modules-nilp-action} $\implies$ \ref{item:irred-graded-modules-L(S)}
Again, nilpotence gives $V^+\ne0$. Set
\[
V_j:=U(\nminus)_{-j}\cdot V^+\quad(j\ge0),\qquad V_j:=0\quad(j<0).
\]
The PBW decomposition implies
\[
\g_i\cdot V_j\subseteq V_{j-i}\qquad(i\in\mathbb Z,\ j\ge0).
\]
Thus $\sum_{j\ge0}V_j$ is a nonzero $\g$-submodule and equals $V$.

Note that  the $\h$-submodule $V^+\cap\nminus V=0$. In fact, if it were nonzero, Lemma~\ref{lem:submod-gen-from-V+} and simplicity would imply
\[
V=U(\nminus)\cdot(V^+\cap\nminus V)\subseteq\nminus V,
\]
and iterating would contradict nilpotence of the action of $\nminus$.

For $j>0$, let
\[
Z_j:=V_j\cap\sum_{0\le k<j}V_k.
\]
We prove $Z_j=0$ by induction on $j$. For $i>0$,
\[
\g_iZ_j\subseteq V_{j-i}\cap\sum_{0\le k<j}V_{k-i}.
\]
If $0<i<j$, the right-hand side is $Z_{j-i}$ and vanishes by induction; if $i\ge j$, it is zero directly. Hence $Z_j\subseteq V^+$. Since $V_j\subseteq\nminus V$ for $j>0$, the preceding intersection result gives $Z_j=0$. Therefore $V=\bigoplus_{j\ge0}V_j$, and placing $V_j$ in degree $-j$ defines a $\g$-module grading.

Let $0\ne Y\subseteq V^+$ be an $\h$-submodule. Lemma~\ref{lem:submod-gen-from-V+} gives $U(\nminus)\cdot Y=V$. Comparing the components of degree zero yields $Y=V^+$, so $V^+$ is simple as an $\h$-module. The canonical graded epimorphism
$\Delta(V^+)\twoheadrightarrow V$,
therefore identifies $V$ with $L(V^+)$. Finally, the classification assertion follows from $L(X)^+=X$.
\end{proof}

\begin{corollary}\label{coro:simple-graded-module-bilinear-form}
Let $V\in\cC$ be a simple $\fg$-module such that the action of $\fn_{\pm}$ is nilpotent. Then $V$ admits a non-degenerate invariant bilinear form $V\otimes V\to \one$ if and only if $V^+\simeq (V^-)^*$ as $\fh$-modules.
\end{corollary}
\begin{proof}
By Theorem~\ref{thm:irred-graded-modules}, we can choose a grading
$ V=\bigoplus_{r=a}^{b}V_r$, with $V_a, V_b\ne0.$ 
Applying the same theorem with the grading reversed, and using Lemma~\ref{lem:L(S)+-is-S}, gives
\[
V^+=V_b,\qquad V^-=V_a.
\]
The dual module is simple, with grading $(V^*)_r=(V_{-r})^*$, so
 $(V^*)^+=(V_a)^*=(V^-)^*.$ 
A nondegenerate invariant bilinear form on $V$ is equivalent to an isomorphism $V\cong V^*$ of $\g$-modules. By the classification in Theorem~\ref{thm:irred-graded-modules}, such an isomorphism exists exactly when $V^+\cong(V^*)^+=(V^-)^*$.
\end{proof}

Next we generalize \cite[Proposition 1.2.5]{Kac}.

\begin{proposition}\label{prop:simple-graded-Lie-alg-bilinear-form}
Assume that $\fg\in\cC$, so $\fg=\oplus_{-n\le k\le m}\fg_k$ for some $m,n\ge 0$ such that $\fg_{-n},\fg_m \ne 0$, and $\fg$ is simple. Then $\fg$ admits a non-degenerate invariant bilinear form $\ttB:\fg\otimes\fg\to\one$ if and only if $\fg_m^* \simeq \fg_{-n}$ as $\fh$-modules. If this happens, then $\ttB(\fg_{h}\otimes\fg_k)=0$ for all $h,k$ such that $h+k\ne m-n$.
\end{proposition}
\begin{proof}
The first statement follows as in loc. cit. using Corollary  \ref{coro:simple-graded-module-bilinear-form} with $V=\fg$. For the second assertion, let $\widetilde B:\g\to\g^*$ be the isomorphism corresponding to the form, and decompose it into homogeneous $\g$-module maps. Every nonzero homogeneous component is an isomorphism by simplicity. A homogeneous isomorphism of degree $s$ must take the highest degree $m$ of $\g$ to the highest degree $n$ of $\g^*$, so $s=n-m$. Hence every nonzero component has this degree, and
\[
\widetilde B(\g_i)\subseteq (\g^*)_{i+n-m}=\g_{m-n-i}^*.
\]
This is exactly the stated orthogonality condition.
\end{proof}

\subsection{The graded category \texorpdfstring{$\mathcal O$}{O}}\label{subsec:cat-O}
Let $\g$ be as in Subsection \ref{subs: graded Lie-alg}. Write
\[
U:=U(\g),\qquad U^0:=U(\h),\qquad
U^+:=U(\nplus),\qquad U^-:=U(\nminus).
\]  The PBW theorem gives an
isomorphism of graded objects
\begin{equation}\label{eq:ordered-pbw}
U^-\otimes U^0\otimes U^+\xrightarrow{\sim}U.
\end{equation}
Let
\[
\Pwt:=\Irr\bigl(\Rep_{\C}(\h)\bigr).
\]

\begin{definition}\label[definition]{def:category-O}
The category $\Ocat^{\mathrm{gr}}(\g)$ consists of graded $U$-modules
$M=\bigoplus_rM_r$ in $\indcat{\C}$ such that
\begin{enumerate}[label=\textup{(\roman*)}]
\item $M$ is finitely generated as a graded $U$-module;
\item $M_r=0$ for $r\gg0$;
\item $M_r$ is finite-dimensional, that is, $M_r\in\C$, for every $r$.
\end{enumerate}
Morphisms preserve the grading.
\end{definition}

\begin{remark}\label{rem:semisimple-subcategory-O}
Condition \textup{(iii)} does not say
that the $\h$-modules $M_r$ are semisimple.  Neither $\C$ nor 
$\Rep_{\C}(\h)$ are assumed semisimple.  The
definition above is independent of an auxiliary choice of semisimple
$\h$-modules and allows nonsplit $\h$-module extensions in individual
degrees.  The arguments below require only that each $M_r$ have finite
length as an $\h$-module.  This follows from $M_r\in\C$, since every chain
of $\h$-submodules of $M_r$ is a chain of subobjects of $M_r$ in $\C$.

We may also construct a restrictive semisimple version.  Namely, let
$\W\subseteq\Rep_{\C}(\h)$ be a full semisimple subcategory that is closed under 
subquotients.  Suppose that
\[
U^-_{-r}\otimes X\in\W
\qquad\text{for every $X\in\W$ and $r\geq0$.}
\]
Define
\[
\Ocat^{\mathrm{gr}}_{\W}(\g)
:=\bigl\{M\in\Ocat^{\mathrm{gr}}(\g):M_r\in\W
\text{ for every $r$}\bigr\}.
\]
Under the noetherian hypothesis in
\cref{prop:O-abelian-noetherian}, this is an abelian full subcategory of
$\Ocat^{\mathrm{gr}}(\g)$, and the argument of
\cref{thm:classification-simple-O} identifies its graded-simple objects
with the modules $L(X,d)$ for
$(X,d)\in\Irr(\W)\times\mathbb Z$.  
\end{remark}

The boundedness condition implies that $\nplus$ acts locally nilpotently.

\begin{proposition}\label[proposition]{prop:O-abelian-noetherian}
If $U(\g)$ is left noetherian, then
$\Ocat^{\mathrm{gr}}(\g)$ is abelian.
\end{proposition}

\begin{proof}
Let $f:M\to N$ be a morphism in $\Ocat^{\mathrm{gr}}(\g)$.  The
kernel is a graded submodule of the finitely generated module $M$, hence is
finitely generated.  It is bounded above, and
 each homogeneous component $(\ker f)_r$ is a subobject of
$M_r\in\C$, so it belongs to $\C$.  The same
properties for the cokernel are immediate.  
\end{proof}

We now prove the noetherianity statement needed for Lie algebras in a
Verlinde category.  Let $p\ge5$ and let $\g\in\Verp$.  Since $\Verp$ is
semisimple, there is a canonical decomposition
\begin{equation}\label{eq:classical-nonclassical}
\g=\g_{\mathrm{cl}}\oplus\g_{\mathrm{nc}},
\end{equation}
where $\g_{\mathrm{cl}}$ is a sum of copies of $\one$ and
$\g_{\mathrm{nc}}$ has no copy of $\one$.  The first summand is an ordinary
finite-dimensional Lie algebra, and it preserves the second summand.

An algebra object $A$ in $\indcat{\C}$ is called left noetherian if every
submodule of a finitely generated left $A$-module is finitely generated.
Right noetherianity is defined analogously.

\begin{theorem}\label{thm:noetherian-enveloping}
Let $\g$ be a finite-dimensional Lie algebra in $\Verp$.  Then $U(\g)$ is
left and right noetherian.  More precisely, it is finite as a left and as a
right module over the ordinary enveloping algebra
$U(\g_{\mathrm{cl}})$.
\end{theorem}

\begin{proof}
By the PBW property and decomposition  \eqref{eq:classical-nonclassical},
\begin{equation}\label{eq:gr-U-classical}
\gr U(\g)
\cong S(\g)
\cong S(\g_{\mathrm{cl}})\otimes S(\g_{\mathrm{nc}}),
\end{equation}
since  symmetric
algebras turn direct sums into tensor products. The symmetric algebra of a non-trivial simple object is finite-dimensional by \cite[Section~5.1]{EtingofPBW}, so  
$S(\g_{\mathrm{nc}})$ is finite-dimensional.  It  follows from
\eqref{eq:gr-U-classical}, by the usual filtered lifting argument, that
$U(\g)$ is finite as a left and as a right
$U(\g_{\mathrm{cl}})$-module.

The ordinary algebra $R:=U(\g_{\mathrm{cl}})$ is left and right noetherian. If $M$ is a finitely generated left $R$-module in $\indcat{\Verp}$, write
\[
M=\bigoplus_{S\in\Irr(\Verp)}S\otimes M_S.
\]
Since $R$ is classical, its action preserves each isotypic component, and each multiplicity space $M_S$ is a finitely generated ordinary $R$-module. Thus every $M_S$ is noetherian. Submodules are described by their multiplicity submodules, so $M$ is noetherian. The right-module argument is identical.

Finally, an algebra that is a finite module over a left or right noetherian algebra is
noetherian on the corresponding side.  This proves the theorem.
\end{proof}

The following is a direct consequence of the previous theorem.

\begin{corollary}\label[corollary]{cor:O-abelian-Verp}
If $\g$ is a finite-dimensional graded Lie algebra in $\Verp$, then
 $\Ocat^{\mathrm{gr}}(\g)$ is abelian.
\end{corollary}

We can extend this beyond $\Ver_p$ as follows. 

\begin{corollary}
   Let $p\ge5$, and suppose that $\C$ admits an exact faithful symmetric
tensor functor $F:\C\to\Verp$. For every finite-dimensional Lie
algebra $\g\in\C$, the algebra $U(\g)$ is left and right
noetherian. If $\g$ is graded, then
$\Ocat^{\mathrm{gr}}(\g)$ is abelian. 
\end{corollary}

\begin{proof}
    The extension of $F$ to ind-objects is exact and faithful, and
\[
 F(U(\g))\cong U(F(\g)).
\]
Moreover, $F(\g)$ has the PBW property, since $F$ preserves the defining
constructions and their filtrations. If $M$ is a finitely generated
$U(\g)$-module, then $F(M)$ is finitely generated over $U(F(\g))$.
By \cref{thm:noetherian-enveloping}, every ascending chain of
submodules of $F(M)$ stabilizes. Exactness and faithfulness therefore
imply that every ascending chain of submodules of $M$ stabilizes.
Every ind-module is the directed union of its submodules generated
by finite subobjects; consequently every submodule of $M$ is finitely
generated. This proves left noetherianity. The right-module argument
is identical, and \cref{prop:O-abelian-noetherian} gives the last
assertion.
\end{proof}

\begin{remark}\label{rk: vermas in O}
For $X\in\Pwt$ and $d\in\mathbb Z$, the Verma module
$\Delta(X,d)$ is an object of
$\Ocat^{\mathrm{gr}}(\g)$, and thus so is every quotient, in particular its simple graded quotient $L(X,d)$. Indeed, by the triangular decomposition  \eqref{eq:ordered-pbw},
\begin{equation}\label{eq:verma-pbw}
\Delta(X,d)\cong U^-\otimes X\langle d\rangle,
\qquad
\Delta(X,d)_{d-r}\cong U^-_{-r}\otimes X.
\end{equation}
Since $U^-_{-r}$ and $X$ belong to $\C$, every homogeneous component of
$\Delta(X,d)$ belongs to $\C$. Thus $\Delta(X,d)$ belongs to $\Ocat^{\mathrm{gr}}(\g)$.
\end{remark}

For any $M$ in this category, induction gives the universal property
\begin{equation}\label{eq:verma-universal-property}
\Hom_{\g}^{\mathrm{gr}}(\Delta(X,d),M)
\cong
\Hom_{\h}\bigl(X,(M_d)^{\nplus}\bigr).
\end{equation}

Next we characterize simple modules in
 $\Ocat^{\mathrm{gr}}(\g)$. This resembles
\cref{thm:irred-graded-modules}, but the hypotheses are slightly different.

\begin{theorem}\label{thm:classification-simple-O}
The graded-simple objects of $\Ocat^{\mathrm{gr}}(\g)$ are, up to
isomorphism, the modules
\[
L(X,d),\qquad (X,d)\in\Pwt\times\mathbb Z.
\]
Moreover, these modules are pairwise nonisomorphic.
\end{theorem}

\begin{proof}
Let $L$ be graded simple.  Choose the largest $d$ with $L_d\ne0$.
Since $L_d\in\C$, it has finite length as an $\h$-module.
Choose a simple $\h$-subobject $X\subseteq L_d$.  Since $\nplus X=0$,
the universal property gives a non-zero map $\Delta(X,d)\to L$, which  is
surjective.  Its image in degree $d$ is $X$, so surjectivity also shows
that $L_d=X$ is simple.  Hence $L\cong L(X,d)$.
Conversely, $L(X,d)$ is graded-simple by
\cref{thm:simple-graded}, and Remark \ref{rk: vermas in O} shows that it belongs
to $\Ocat^{\mathrm{gr}}(\g)$. The highest degree and its $\h$-module
recover $(X,d)$, proving uniqueness.
\end{proof}

We conclude this subsection with a discussion on gradability of simple modules over finite graded algebras. The following lemma resembles Theorem \ref{thm:irred-graded-modules}, but here we assume the algebra is in $\mathcal C$, and that $\mathcal C$ is a fusion category.

\begin{lemma}\label{lem:gradability}
Let $B=\bigoplus_{n\in\mathbb Z}B_n$ be a $\mathbb Z$-graded
algebra object in a finite tensor category $\C$.    Then every simple 
$B$-module $V$ admits a grading.  Any two gradings on $V$ differ
by a shift.
\end{lemma}

\begin{proof}
For each $t\in\mathbb G_m$, the grading of $B$ defines an algebra
automorphism
\[
\phi_t:B\longrightarrow B,
\qquad
\phi_t|_{B_n}=t^n\id_{B_n}.
\]
If $V$ is a $B$-module, let ${}^{\phi_t}V$ denote the module with the same
underlying object as $V$ and with action twisted by $\phi_t$. 
Let $P$ be a projective generator of $\C$. Then $B\otimes P$ is
a projective generator of the category of finite $B$-modules, because
\[
 \Hom_B(B\otimes P,M)\cong\Hom_{\C}(P,M).
\]
Indeed, the right side is exact, and an epimorphism
$P^{\oplus n}\twoheadrightarrow M$ induces an epimorphism
$(B\otimes P)^{\oplus n}\twoheadrightarrow M$.
Since $B\otimes P$ has finite length, the category of finite
$B$-modules has only finitely many isomorphism classes of simples.
Twisting defines an action of $\Bbbk^\times$ on this finite set. Since $\Bbbk$ is algebraically closed, $\Bbbk^\times$ is a divisible abelian group. Its image in a finite permutation group is therefore a finite divisible abelian group and must be trivial. Consequently, ${}^{\phi_t}V\cong V$ for every $t\in\Bbbk^\times$.

We choose these isomorphisms compatibly as $t$ varies. Consider the algebraic group
\[
\widetilde T_V
=
\bigl\{
(t,F):
t\in\mathbb G_m,\ 
F:V\xrightarrow{\sim}{}^{\phi_t}V
\text{ is a $B$-module isomorphism}
\bigr\}.
\]
The spaces $\Hom_B(V,{}^{\phi_t}V)$ are the kernels of a polynomial
family of linear maps between fixed finite-dimensional morphism
spaces. Every fiber has dimension one, by Schur's lemma and the
isomorphisms ${}^{\phi_t}V\cong V$ already proved. These kernels
therefore form a line bundle over $\mathbb G_m$. Its complement of
the zero section is precisely $\widetilde T_V$, since every nonzero
map in such a fiber is an isomorphism. Thus $\widetilde T_V$ is a
smooth $\mathbb G_m$-torsor over $\mathbb G_m$.

Projection onto the first coordinate gives a surjective homomorphism
\(
\pi:\widetilde T_V\longrightarrow\mathbb G_m.
\)
The kernel consists of the $B$-module automorphisms of $V$. Since $V$ is
simple, Schur's lemma gives
$\ker\pi=\operatorname{Aut}_B(V)=\mathbb G_m$, acting by scalar automorphisms. Taking the identity component $\widetilde T_V^\circ$, we obtain an exact sequence
\[
1\longrightarrow\mathbb G_m
\longrightarrow\widetilde T_V^\circ
\xrightarrow{\pi}\mathbb G_m
\longrightarrow1.
\]

We claim that $\widetilde T_V^\circ$ is a torus. Its unipotent radical maps
trivially to the quotient $\mathbb G_m$ and hence lies in the kernel,
which is also a torus. Therefore, the unipotent radical is trivial, so
$\widetilde T_V^\circ$ is reductive. Its derived subgroup also maps
trivially to the abelian quotient and hence lies in the central torus
$\ker\pi$. Since a semisimple group cannot be a nontrivial subgroup of a
torus, the derived subgroup is trivial. Thus
$\widetilde T_V^\circ$ is a torus.

Since $\Bbbk$ is algebraically closed, such an extension splits, so there is a homomorphism
\[
s:\mathbb G_m\longrightarrow\widetilde T_V^\circ,
\qquad
s(t)=(t,F_t),
\]
such that $\pi\circ s=\id_{\mathbb G_m}$. The maps $F_t$ define an action
of $\mathbb G_m$ on the underlying object of $V$, and hence a weight
decomposition
\[
V=\bigoplus_{m\in\mathbb Z}V_m,
\qquad
F_t|_{V_m}=t^m\id_{V_m}.
\]
It remains to check that this weight decomposition is a $B$-module
grading. If $b\in B_n$ and $v\in V_m$, then the fact that
$F_t:V\to{}^{\phi_t}V$ is a $B$-module map gives
\[
F_t(bv)
=
\phi_t(b)F_t(v)
=
t^n b\,t^m v
=
t^{n+m}bv.
\]
Therefore,
\(
B_nV_m\subseteq V_{n+m},
\)
as desired.

Finally, suppose that two gradings of $V$ give compatible families
$(F_t)$ and $(F'_t)$. For each $t$, the automorphism
$F_t^{-1}F'_t:V\longrightarrow V$
is $B$-linear and hence scalar by Schur's lemma. Compatibility with
multiplication in $\mathbb G_m$ shows that these scalars form a character
of $\mathbb G_m$, so they have the form $t\mapsto t^d$ for some
$d\in\mathbb Z$. Thus
$F'_t=t^dF_t,$ 
which shifts every weight by the same integer $d$. 
\end{proof}

The category $\Ocat^{\mathrm{gr}}(\g)$ need not be a highest
weight category in the sense of \cite{CPS}.  For example, take
$\C=\operatorname{Vec}_{\Bbbk}$ and let $\g=\Bbbk f$ be one-dimensional abelian,
concentrated in degree $-1$.  Then
\[
\Delta(\one,d)=\Bbbk[f]\langle d\rangle
\]
has the strictly descending chain
\[
\Delta(\one,d)\supset f\Delta(\one,d)\supset
f^2\Delta(\one,d)\supset\cdots.
\]
Thus $\Ocat^{\mathrm{gr}}(\g)$ is not artinian.

\begin{question}\label[question]{q:highest-weight-truncations}
For integers $a\le b$, let $\Ocat^{\mathrm{gr}}_{[a,b]}(\g)$ be the full subcategory of modules supported in degrees $a,\ldots,b$. Under what additional hypotheses is this a highest weight category, with standard objects given by the largest quotients of the Verma modules supported in $[a,b]$?
\end{question}

\subsection{Shapovalov maps and generic graded irreducibility}
We follow the notation of Subsection \ref{subs: graded Lie-alg}.
Let
\[
\mathfrak X(\h)
:=
\Hom_{\mathrm{Lie}}(\h,\one)
=
\Hom_{\C}(\h/[\h,\h],\one)
\]
be the vector space of characters of $\h$. For
$\lambda\in\mathfrak X(\h)$ and $i>0$, define
\[
\beta_{\lambda,i}:\g_i\otimes\g_{-i}\longrightarrow\one,
\qquad
\beta_{\lambda,i}(x,y)=\lambda([x,y]).
\]
We extend the pairings \(\beta_{\lambda,i}\) to a pairing $\beta_\lambda:\nplus\otimes\nminus\longrightarrow\one$
by setting
\[
\left.\beta_\lambda\right|_{\g_i\otimes\g_{-j}}
 =
 \begin{cases}
   \beta_{\lambda,i},& i=j,\\
   0,& i\neq j.
 \end{cases}
\]
Thus homogeneous components of different absolute degrees are
orthogonal.

\begin{definition}\label[definition]{def:nondegenerate-grading}
The graded Lie algebra $\g$ is \emph{nondegenerate} if there is a character
$\lambda\in\mathfrak X(\h)$ for which every pairing
$\beta_{\lambda,i}$ is nondegenerate, or equivalently $\beta_{\lambda}$ is nondegenerate.
\end{definition}

The locus
\[
 \mathfrak X(\h)_{\mathrm{nd}}
 :=\{\lambda:\beta_{\lambda,i}\text{ is nondegenerate for all }i>0\}
\]
is Zariski open. Indeed, for each $i$ such that
$\g_i\cong\g_{-i}^*$, choose an isomorphism
$c_i:\g_i\xrightarrow{\sim}\g_{-i}^*$ and define
\[
 D_i(\lambda)
 :=
 \det_{\Bbbk}\!\left(
 L_{c_i^{-1}\beta_{\lambda,i}}:
 \End_{\C}(\g_i)\longrightarrow\End_{\C}(\g_i)
 \right).
\]
This is a polynomial in $\lambda$, and it is nonzero precisely when
$\beta_{\lambda,i}$ is an isomorphism. If
$\g_i\not\cong\g_{-i}^*$, the locus is empty. Since only finitely many
degrees occur, $\mathfrak X(\h)_{\mathrm{nd}}$ is Zariski open.
Thus $\g$ is nondegenerate exactly when this open set is nonempty.

\begin{remark}
Strong semisimplicity gives a sufficient condition for the grading
to be nondegenerate. Assume that $\operatorname{char}\Bbbk\neq2$ and that
$\g$ is strongly semisimple. Then the degree derivation is $\ad(h)$ for
some $h\in\h$ by Theorem \ref{thm:all-derivations-inner}, and the Killing form $B$ has degree zero, i.e., $B(\g_i,\g_j)=0$ when $i+j\ne0.$ The character
\(
\lambda_h(z):=B(h,z)
\)
satisfies
\(
\beta_{\lambda_h,i}(x,y)=iB(x,y).
\)
Consequently, the grading is nondegenerate in characteristic zero. It is
also nondegenerate in characteristic $p$ provided that
$\g_i=0$ for every non-zero $i\in p\mathbb Z.$
\end{remark}

Let $X$ be an $\h$-module. 
For $\lambda\in\mathfrak X(\h)$, let $\one_\lambda$ be the corresponding
one-dimensional $\h$-module and set
\[
X_\lambda:=X\otimes\one_\lambda.
\]
Notice that if $X$ is simple, then so is $X_\lambda$.

The triangular decomposition of $\g$, together with the PBW theorem, gives
an isomorphism of graded objects
\[
m:U^-\otimes U(\h)\otimes U^+
\xrightarrow{\sim}U(\g).
\]
The associated Harish--Chandra projection is
\[
\HC:U(\g)\longrightarrow U(\h)
\qquad
\HC
=
(\varepsilon^-\otimes\id_{U(\h)}\otimes\varepsilon^+)\circ m^{-1},
\]
that is, we take the component with no factor from $\nminus$ or $\nplus$.

\begin{definition}
For $r>0$, define the (degree $r$) \emph{Shapovalov map}
\begin{equation}\label{eq:shapovalov-map}
\mathsf S_{r,X}(\lambda):
U^-_{-r}\otimes X_{\lambda}
\longrightarrow
(U^+_r)^*\otimes X_{\lambda}
\end{equation}
by
\begin{equation}\label{eq:shapovalov-formula}
\bigl(\mathsf S_{r,X}(\lambda)(u\otimes x)\bigr)(v)
:=\HC(vu)\cdot x,
\qquad x\in X_{\lambda}, \, u\in U^-_{-r}, \, v\in U^+_r.
\end{equation}
\end{definition}

From  now on, assume $X$ is simple. In the rest of this subsection,  we extend the classical theory of Shapovalov maps to our context. 

\begin{lemma}\label[lemma]{lem:shapovalov-detects-simplicity}
If $\mathsf S_{r,X}(\lambda)$ is injective for every $r>0$, then
$\Delta(X_\lambda,d)$ is graded-simple.
\end{lemma}

\begin{proof}
Let $N=\oplus_{j\le d} N_j$ be a non-zero proper graded submodule. Since the simple object 
$X_\lambda=\Delta(X_\lambda,d)_d$ generates the Verma module, $N_d=0$.
Choose $r>0$ minimal such that $N_{d-r}\ne0$. Then
$\nplus\cdot N_{d-r}=0$, so $\mathsf S_{r,X}(\lambda)$ vanishes on the
non-zero subobject $N_{d-r}$. This contradicts injectivity.
\end{proof}

The Shapovalov maps depend algebraically on the character $\lambda$, and
we will use their determinants to study injectivity for generic
$\lambda$.  This requires evaluating them at the generic point of
$\mathfrak X(\h)$, which is defined over the function field
$\Bbbk (\mathfrak X(\h))$ rather than over $\Bbbk$.  The resulting maps therefore
live in a scalar extension of $\C$. 

For a field extension $K/\Bbbk$, we use the abelian scalar extension
$\C_K$. It is again a symmetric tensor category, the scalar-extension
functor is exact, and
\[
 \Hom_{\C_K}(A_K,B_K)
 \cong K\otimes_{\Bbbk}\Hom_{\C}(A,B);
\]
see \cite[Theorem~1.3.18 and Proposition~2.2.2]{StalderScalar}. 

When $\C$ is semisimple, writing objects of \(\C_K\) as
 $ A\simeq\bigoplus_{S} S_K\otimes_K V_S$  and 
 $B\simeq\bigoplus_{S} S_K\otimes_K W_S,$ 
a morphism \(f:A\to B\) is determined by linear maps
$ f_S:V_S\longrightarrow W_S$
between the corresponding multiplicity spaces.  When every \(f_S\) is
square, we define
$ \det(f):=\prod_S\det(f_S).$ 
Although this determinant depends on the chosen decompositions and bases
up to a non-zero scalar, its vanishing does not.

We take the filtration of $U^+$ and $U^-$ given by the number of Lie
algebra factors in a PBW monomial.  This induces filtrations on the source
and target of each Shapovalov map \eqref{eq:shapovalov-map}. We will need the following technical results.

\begin{lemma}
\label[lemma]{lem:shapovalov-leading-term}
Let $\lambda\in\mathfrak X(\h)_{\mathrm{nd}}$. On the PBW length-\(q\) piece, write
$ u=x_1\cdots x_q\in U^-_{-r}$ and $v=y_1\cdots y_q\in U^+_r.$ The coefficient of $t^q$ in
$\mathsf S_{r,X}(t\lambda)$ is given by
\begin{equation}\label{eq:fock-pairing}
(y_1\cdots y_q,x_1\cdots x_q)
\longmapsto
\sum_{\sigma\in S_q}
\prod_{j=1}^q
\beta_{\lambda}(y_j,x_{\sigma(j)}).
\end{equation}
\end{lemma}

\begin{proof}
By \eqref{eq:shapovalov-formula}, the matrix coefficient is
obtained by applying \(\HC\) to \(vu\).  We move the positive generators
\(y_1,\ldots,y_q\) past the negative generators
\(x_1,\ldots,x_q\).  A contribution of degree \(q\) in \(t\) can arise
only by pairing every \(y_j\) with one of the \(x_i\) and letting the
resulting degree-zero commutator act on \(X_{t\lambda}\) through the
character \(t\lambda\).  Then the possible complete pairings are indexed by
\(\sigma\in S_q\), and we obtain  the formula in the statement. 
\end{proof}

\begin{remark}\label{rem:generic-shapovalov-isomorphism} 
Assume $\operatorname{char}\Bbbk=0$ and
$\mathfrak X(\h)_{\mathrm{nd}}\ne\varnothing$. Then, for every $r>0$,
the Shapovalov map \(\mathsf S_{r,X}(\lambda)\) is an isomorphism at the generic point
\(\lambda\) of \(\mathfrak X(\h)\), after extending scalars to
\(K=\Bbbk(\mathfrak X(\h))\).

Indeed, fix $\lambda_0\in\mathfrak X(\h)_{\mathrm{nd}}$ and set
\[
 A_r=U^-_{-r}\otimes X,
 \qquad
 B_r=(U^+_r)^*\otimes X.
\]
PBW symmetrization and the perfect pairing $\beta_{\lambda_0}$ give
an isomorphism $c_r:A_r\xrightarrow{\sim}B_r$. Define
\[
 D_r(\lambda)
 :=
 \det_{\Bbbk}\!\left(
 L_{c_r^{-1}\mathsf S_{r,X}(\lambda)}:
 \End_{\C}(A_r)\longrightarrow\End_{\C}(A_r)
 \right).
\]
Thus $D_r(\lambda)\ne0$ precisely when
$\mathsf S_{r,X}(\lambda)$ is an isomorphism.

To see that $D_r$ is not identically zero, restrict to
$\lambda=t\lambda_0$ and decompose $A_r$ and $B_r$ by PBW length.
The block pairing terms of lengths $p$ and $q$ has degree at most
$\min(p,q)$ in $t$, since each occurrence of $\lambda$ uses one
positive and one negative factor. When $p=q$, its leading coefficient
is the full contraction induced by $\beta_{\lambda_0}$, multiplied by
$q!$ and tensored with $\id_X$, and is therefore invertible. After
rescaling the target summand of length $q$ by $t^{-q}$, the matrix at
$t=\infty$ is triangular with invertible diagonal blocks. Hence
$\mathsf S_{r,X}(t\lambda_0)$ is invertible over $\Bbbk(t)$, so
$D_r$ is not the zero polynomial.

Consequently,
\[
 \{\lambda\in\mathfrak X(\h):
   \mathsf S_{r,X}(\lambda)\text{ is an isomorphism}\}
\]
is a nonempty Zariski open subset of $\mathfrak X(\h)$.
\end{remark}

We prove now our first main result in this subsection.  Let \(\eta\) be the generic point of \(\mathfrak X(\h)\), and let
$  K=\Bbbk(\mathfrak X(\h))$ 
be its function field.  Denote by \(\C_K\) the scalar extension of
\(\C\) from \(\Bbbk\) to \(K\).  Viewing \(\eta\) as a \(K\)-valued
character of \(\h_K\), set
\[ X_\eta:=X_K\otimes\one_\eta.
\]
For completeness, a simple $\h$-module $X$ remains simple after
this scalar extension. Viewed over $\Bbbk$, $K\otimes_{\Bbbk}X$
is a direct sum of copies of $X$, and
$\End_{\h}(X)=\Bbbk$. Any finite underlying subobject is contained
in finitely many copies of $X$, so the $\h$-submodule it generates
is still finite. Thus $\h$-subobjects of this isotypic ind-object
correspond to $\Bbbk$-subspaces of $K$. A $K$-linear subobject
corresponds to a subspace stable under multiplication by $K$,
and is consequently zero or the whole object. Hence $X_K$, and
therefore $X_\eta$, is simple.

\begin{theorem}[Generic irreducibility in characteristic zero]\label{thm:generic-irreducibility-char0} Assume that $\operatorname{char}\Bbbk=0$.
Let $\g$ be a nondegenerate graded Lie algebra, let $X\in\Pwt$, and let
$d\in\mathbb Z$.  Then
\begin{enumerate}[leftmargin=*,label=\rm{(\roman*)}]
    \item\label{item:generic-irreducibility-char0-graded-simple} $\Delta(X_\eta,d)$ 
is graded-simple as a \(\g_K\)-module.
\item\label{item:generic-irreducibility-char0-singularvectors} For every \(N\geq1\), there is a nonempty Zariski open subset
$  U_N\subseteq\mathfrak X(\h)$ 
such that, for every \(\lambda\in U_N\), the Verma module
\(\Delta(X_\lambda,d)\) has no non-zero singular vectors in degrees
\(d-1,\ldots,d-N\).
\item\label{item:generic-irreducibility-char0-countableunion} If $\Bbbk$ is uncountable, there is a countable union $Z$ of proper closed
subsets of $\mathfrak X(\h)$ such that $\Delta(X_\lambda,d)$ is
graded-simple for $\lambda\notin Z$.  In particular, such characters exist.
\item\label{item:generic-irreducibility-char0-simple} Every graded-simple Verma module $\Delta(X_\lambda,d)$ is simple as a $\g$-module.
\end{enumerate}
\end{theorem}

\begin{proof} Part \ref{item:generic-irreducibility-char0-graded-simple} follows from Remark \ref{rem:generic-shapovalov-isomorphism} and 
\cref{lem:shapovalov-detects-simplicity}. For part \ref{item:generic-irreducibility-char0-singularvectors}, let \(N\geq1\) and set
\[
 U_N:=\bigcap_{r=1}^N
 \{\lambda\in\mathfrak X(\h):D_r(\lambda)\ne0\}.
\]
By Remark \ref{rem:generic-shapovalov-isomorphism}, every $D_r$ is a
nonzero polynomial, so $U_N$ is a nonempty Zariski open subset of
$\mathfrak X(\h)$.  For
\(\lambda\in U_N\), the maps
\(\mathsf S_{r,X}(\lambda)\) are injective for \(1\leq r\leq N\), so
\(\Delta(X_\lambda,d)\) has no non-zero singular vectors in degrees
\(d-1,\ldots,d-N\).  

We now prove part \ref{item:generic-irreducibility-char0-countableunion}. Assume \(\Bbbk\) is uncountable. A finite-dimensional affine space over
\(\Bbbk\) cannot be the union of countably many proper closed subsets.
Therefore
$  \bigcap_{N\geq1}U_N\neq\varnothing.$ 
For every \(\lambda\) in this intersection, all the Shapovalov maps are
injective, and hence
\(\Delta(X_\lambda,d)\) is graded-simple by
\cref{lem:shapovalov-detects-simplicity}.

Lastly, we look at part \ref{item:generic-irreducibility-char0-simple}. Let $M=\Delta(X_\lambda,d)$ be graded-simple. By Lemma~\ref{lem:L(S)+-is-S}, its $\nplus$-invariants are exactly its simple highest-weight submodule $X_\lambda$. If $0\ne W\subseteq M$ is an ungraded $\g$-submodule, choose a nonzero finite subobject $A\subseteq W$. The image of $A$ is contained in a finite sum of homogeneous components of $M$. Since the grading is bounded above and $\nplus$ raises degrees, a sufficiently long iterated action of $\nplus$ on $A$ is zero. Its last nonzero image is therefore a nonzero subobject of $W^+$. Thus $0\ne W^+\subseteq X_\lambda$, and simplicity of $X_\lambda$ gives $X_\lambda\subseteq W$. Since the top generates $M$, we have $W=M$.
\end{proof}

Assume now that $\C=\Verp$ and that the grading of $\g$ is nondegenerate. In the classical setting, Verma modules in characteristic $p$  often have singular vectors
coming from central $p$-th powers.  We now define the baby
Verma modules in $\Verp$. Write $\nplus_{\mathrm{cl}}$ and $\nminus_{\mathrm{cl}}$ for the
classical parts of $\nplus$ and $\nminus$. Suppose that 
\begin{equation}\label{eq:p-nilpotent-assumption}
(\ad x)^p=0
\quad\text{ for every }
x\in\nplus_{\mathrm{cl}}\cup\nminus_{\mathrm{cl}}.
\end{equation}
Then $x^p$ is central in $U(\g)$, since
$[x^p,y]=(\ad x)^p(y)$.    Let $J^-\subseteq U^-$ be the
ideal generated by $x^p$ for all  homogeneous  $x\in\nminus_{\mathrm{cl}}$. We define the \emph{baby Verma module} as
\begin{equation}\label{eq:baby-verma}
\Delta_b(X_\lambda,d)
:=\Delta(X_\lambda,d)/J^-\Delta(X_\lambda,d).
\end{equation}

Let $J^+\subseteq U^+$ be generated by $x^p$ for homogeneous
$x\in\nplus_{\mathrm{cl}}$.  This ideal acts by zero on $\Delta_b(X_\lambda,d)$: its generators are
central in $U(\g)$ and annihilate the highest weight subobject.  Set
$\overline U^\pm:=U^\pm/J^\pm.$ 
Formula \eqref{eq:shapovalov-formula} therefore induces baby Shapovalov maps
\begin{equation}\label{eq:baby-shapovalov-map}
\mathsf S^b_{r,X}(\lambda):
\overline U^-_{-r}\otimes X_{\lambda}
\longrightarrow
(\overline U^+_r)^*\otimes X_{\lambda}.
\end{equation}
As in \cref{lem:shapovalov-detects-simplicity}, injectivity of all these
maps implies graded simplicity of the baby Verma module.

Fix $\lambda_0\in\mathfrak X(\h)_{\mathrm{nd}}$. Then $Z(\g)\cap\nplus=Z(\g)\cap\nminus=0$, since every central homogeneous subobject of nonzero degree lies in the radical of the corresponding pairing $\beta_{\lambda_0,i}$. For  $x,y\in\nplus_{\mathrm{cl}}$, respectively or $x,y\in\nminus_{\mathrm{cl}}$, the Jacobson identity says that $(x+y)^p-x^p-y^p$ is a Lie element in the same nilpotent part. By \eqref{eq:p-nilpotent-assumption} it is central in $U(\g)$, and hence is zero. Therefore the map $x\mapsto x^p$ on each classical part is Frobenius-semilinear. Since the generators $x^p$ are homogeneous for both the PBW filtration and the $\mathbb Z$-grading, the associated graded ideals of $J^\pm$ are identified with the ideals generated by the corresponding $p$-th
powers.  Thus the associated graded positive and negative parts are
\begin{equation}\label{eq:baby-symmetric-algebra}
S_b(\nplus):=S(\nplus)\big/(x^p:x\in\nplus_{\mathrm{cl}}),
\qquad
S_b(\nminus):=S(\nminus)\big/(x^p:x\in\nminus_{\mathrm{cl}}).
\end{equation}
They are finite-dimensional.  Indeed, the classical polynomial components 
are truncated at degree $p$, while the symmetric algebra of each
nonclassical part is finite-dimensional in $\Verp$.

\begin{lemma}\label[lemma]{lem:reduced-contraction-pairing}
Let $\lambda_0\in\mathfrak X(\h)_{\mathrm{nd}}$.  The pairing
$\beta_{\lambda_0}$ induces a nondegenerate pairing
\[
S_b(\nplus)\otimes S_b(\nminus)\longrightarrow\one.
\]
\end{lemma}

\begin{proof}
We split $\nplus$ and $\nminus$ into their classical and nonclassical parts.
On the classical parts, the pairing is the usual perfect pairing between
truncated polynomial rings: monomials have exponents between $0$ and
$p-1$, so the relevant factorials are non-zero in $\Bbbk$.

On the nonclassical parts, the pairing identifies the negative part with
the dual of the positive part.  Their symmetric algebras have finite
length.  The natural action of the corresponding Weyl algebra on the symmetric
algebra identifies the Weyl algebra with its full endomorphism algebra by
\cite[Corollary~3.4(ii)]{EtingofWeyl}.  Equivalently, the contraction
pairing of the two symmetric algebras is nondegenerate.  We tensor the
classical and nonclassical pairings to obtain the desired nondegenerate pairing.
\end{proof}

We now prove our second main result of this subsection.

\begin{theorem}[Generic irreducibility in $\Verp$]\label{thm:generic-irreducibility-verp}
Assume $\C=\Verp$ and  $p\ge5$.
Let $\g$ be a nondegenerate graded Lie algebra satisfying \eqref{eq:p-nilpotent-assumption}, let $X\in\Pwt$, and let
$d\in\mathbb Z$.
 Then there is a nonempty Zariski
open subset $U\subseteq\mathfrak X(\h)$ such that
$\Delta_b(X_\lambda,d)$ 
is graded-simple for every $\lambda\in U$.
\end{theorem}

\begin{proof}
By \cref{lem:shapovalov-leading-term} and
\cref{lem:reduced-contraction-pairing},   the determinant of
each map \eqref{eq:baby-shapovalov-map} is a non-zero polynomial in
$\lambda$. Only finitely many degree components of $\overline U^\pm$ are nonzero. Hence the intersection of the nonvanishing loci of the finitely many relevant Shapovalov determinants is a nonempty Zariski open subset. On this subset all baby Shapovalov maps are isomorphisms. We then apply the baby version of
\cref{lem:shapovalov-detects-simplicity} to prove the statement.
\end{proof}

\subsection{Example:   restricted representations of 
\texorpdfstring{$\mathfrak{sl}(\one\oplus L_3)$ in $\Ver_5^+$}
{sl(1+L3) in Ver5+}}
\label{subsec:sl-one-plus-L3}

We now illustrate the preceding constructions in a single example.  We
compute the baby Verma modules and their simple quotients, use Shapovalov
maps to determine exactly which baby Verma modules are simple, and then
describe the projective covers and blocks.  In particular, the example
shows that a category built from highest-weight modules need not itself be
a highest weight category.

Assume that $\operatorname{char}\Bbbk=5$, put
$X=L_3\in\Ver_5^+$, and let
\[
 \g=\mathfrak{sl}(\one\oplus X).
\]
The block decomposition of $\one\oplus X$ gives a $\mathbb Z$-grading
\[
 \g=\g(-1)\oplus\g(0)\oplus\g(1).
\]
Let
\(
h:=2\id_{\one}\oplus\id_X
\in\End(\one\oplus X).
\)
Thus $h$ acts by the scalar $2$ on $\one$ and by the scalar $1$ on $X$.
Since $\dim(X)=3$ in $\Ver_5^+$ and $\operatorname{char}\Bbbk=5$, we have
\[
\tr(h)=2\dim(\one)+\dim(X)=2+3=0.
\]
Hence $h\in\g=\mathfrak{sl}(\one\oplus X)$.
Moreover,
\(
 \g(0)=\Bbbk h\oplus X,
\)
and $\g(1)$ and $\g(-1)$ are the upper and lower copies of
$X$, respectively.  Thus $[h,\g(i)]=i\g(i)$.

The element $h^5-h$ is central and primitive in $U(\g)$.  We consider the
restricted enveloping algebra $u(\g):=U(\g)/(h^5-h)$. 
Set $\mathfrak s:=\mathfrak{sl}(X)\subseteq\g(0)$. For $M\in\Ver_5^+$, we define $a_M:\mathfrak s\otimes M\to M$ on the isotypic summands of $M$ using the trivial action on $\one$ and the natural action on $X$. This is natural in $M$ and compatible with tensor products. In this example we consider the full subcategory
\[
 \mathcal R:=\left\{(M,\rho_M)\in\Rep_{\Ver_5^+}(u(\g)):\;\rho_M|_{\mathfrak s\otimes M}=a_M\right\}.
\]

\subsubsection{Baby Verma modules}

For
$n\in\mathbb F_5$ and $Z\in\{\one,X\}$, let $Z_n$ be the simple
$\g(0)$-module on which $h$ acts by $n$ and the copy of $X$ acts
canonically. The corresponding highest-weight
baby Verma module is
\begin{equation}\label{eq:sl1L3-Verma-definition}
 \Delta(Z,n):=u(\g)\otimes_{u(\mathfrak b^+)}Z_n.
\end{equation}
Since
\[
 S(X)=\one\oplus X\oplus\one,
 \qquad X\otimes X=\one\oplus X,
\]
we get the graded objects
\begin{equation}\label{eq:sl1L3-Verma-characters}
 \begin{aligned}
\Delta(\one,n)&=\one_n\oplus X_{n-1}\oplus\one_{n-2},
\\
\Delta(X,n)&=X_n\oplus(\one_{n-1}\oplus X_{n-1})\oplus X_{n-2},
 \end{aligned}
\end{equation}
where subscripts record the $h$-weight, read modulo $5$.

The algebra $u(\g)$ is finite and graded.  Hence
\cref{lem:gradability} shows that every simple object of $\mathcal R$
admits a grading.  Its highest homogeneous component is one of the ten
simple $u(\g(0))$-modules $Z_n$, with
$Z\in\{\one,X\}$ and $n\in\mathbb F_5$.  Therefore the simple quotients
of the ten baby Verma modules in
\eqref{eq:sl1L3-Verma-characters} exhaust the simple objects of
$\mathcal R$.  We determine these quotients using Shapovalov maps.

\subsubsection{Shapovalov maps and simple modules}
\label{subsec:sl-one-plus-L3-shapovalov}

 The Shapovalov calculation is particularly short: the
positive and negative parts are nonclassical, and only two degrees occur.
Indeed,
\[
 \g_{-1}\cong X,
 \qquad
 \h=\g_0\cong\Bbbk h\oplus X,
 \qquad
 \g_1\cong X.
\]
Since $\Hom_{\C}(X,\one)=0$, a character of $\h$ is determined by its
value on $h$.  For $a\in\Bbbk$, let $\lambda_a$ be the character given by
\[
 \lambda_a(h)=a,
 \qquad
 \lambda_a|_X=0.
\]
Thus $\mathfrak X(\h)\cong\mathbb A^1$.  The pairing
$\beta_{\lambda_a,1}:X\otimes X\to\one$ is a non-zero scalar multiple of
$a\,\ev_X$.  Consequently,
\[
 \mathfrak X(\h)_{\mathrm{nd}}=\mathbb A^1\setminus\{0\},
\]
and in particular the grading is nondegenerate.

For the Shapovalov calculation, we allow arbitrary $a\in\Bbbk$ and place
the highest-weight object $Z_a$ in degree $d$; the grading shift is
suppressed in the notation for the baby Verma modules above.  Recall 
$  S(X)=\one\oplus X\oplus\one,$ 
and the summands have degrees $0,1,2$.  Hence
\[
 U^-_{-1}\cong U^+_1\cong X,
 \qquad
 U^-_{-2}\cong U^+_2\cong\one,
\]
and there are no Shapovalov maps in higher degrees.  Moreover,
$\nplus$ and $\nminus$ are abelian and have no classical summands, so $J^+=J^-=0$.
Thus the ordinary and baby Shapovalov maps coincide in this example.

Using \eqref{eq:shapovalov-formula} and $X\otimes X=\one\oplus X$, we obtain the following determinants. Here $f\doteq g$ means equality up to a non-zero scalar independent of $a$.

\begin{center} 
\vspace{10pt}
\renewcommand{\arraystretch}{1.2} \begin{tabular}{ccl} \toprule highest-weight object & degree & Shapovalov determinant\\ \midrule $\one_a$ & $1$ & $\det\mathsf S_{1,\one}(\lambda_a)\doteq a$\\ $\one_a$ & $2$ & $\det\mathsf S_{2,\one}(\lambda_a)\doteq a(a-2)$\\ $X_a$ & $1$ & $\det\mathsf S_{1,X}(\lambda_a)\doteq(a-3)(a-4)$\\ $X_a$ & $2$ & $\det\mathsf S_{2,X}(\lambda_a)\doteq(a-3)(a-4)$\\ \bottomrule \end{tabular} 
\vspace{10pt}
\end{center}

We include here a brief explanation of this table. 
First, $\mathsf S_{1,\one}(\lambda_a)$ is a scalar endomorphism of $X$. Since $\HC(yx)=[y,x]$ for $y\in\g_1$ and $x\in\g_{-1}$, only the $h$-component of $[y,x]$ acts on $\one_a$, giving the factor $a$.
Here, we identify $\mathfrak{sl}(X)$ with $\mathfrak{sl}(L_2)$ through the latter's action on $X=S^2L_2$. We normalize the Casimir using the trace form on $L_2$.
For \(Z=X\), the degree-one source is
$ X\otimes X=\one\oplus X.$ 
The map \(\mathsf S_{1,X}(\lambda_a)\) is
\(\g(0)\)-equivariant, so it acts by a scalar on each summand.  More
precisely, up to an overall non-zero scalar, it is
$\frac a2\id+\Omega,$ 
where \(\Omega\) is the tensor Casimir for the copy of \(X\) in
\(\g(0)\).  The Casimir acts by \(1\) on the \(\one\)-summand and by
\(-2\) on the \(X\)-summand.  Therefore, 
\[
 \left.\mathsf S_{1,X}(\lambda_a)\right|_{\one}
 \doteq
 \left(\frac a2+1\right)\id_{\one}
 \doteq(a-3)\id_{\one}, \quad \text{and} \quad 
 \left.\mathsf S_{1,X}(\lambda_a)\right|_X
 \doteq
 \left(\frac a2-2\right)\id_X
 \doteq(a-4)\id_X.
\]
Thus $  \det\mathsf S_{1,X}(\lambda_a)
 \doteq(a-3)(a-4).$

For the degree-two maps, $\mathsf S_{2,Z}(\lambda_a)$ is a scalar endomorphism of $Z$.
Moving two positive generators past two negative generators in
\eqref{eq:shapovalov-formula} gives
\[
 \mathsf S_{2,Z}(\lambda_a)
 \doteq
 \bigl(a^2-2a-2c_Z\bigr)\id_Z,
\]
where \(c_Z\) is the scalar by which the Casimir of the copy of
\(X\) in \(\g(0)\) acts on \(Z\).  Since
$ c_{\one}=0,$ and  $c_X=4$,
we obtain
$  \det\mathsf S_{2,\one}(\lambda_a)
 \doteq a(a-2)$ 
and 
 $\det\mathsf S_{2,X}(\lambda_a)
 \doteq(a-3)(a-4).$

It follows that
\begin{align*}
 \Delta(\one_a,d)\text{ is graded-simple }
 &\Longleftrightarrow a\notin\{0,2\},
&
 \Delta(X_a,d)\text{ is graded-simple }
 &\Longleftrightarrow a\notin\{3,4\}.
\end{align*}

At the exceptional values, the first singular subobjects occur as follows:
\[
\begin{array}{c|c|c}
 Z&a&\text{first singular subobject}\\
\hline
 \one&0&X\text{ in degree }d-1\\
 \one&2&\one\text{ in degree }d-2\\
 X&3&\one\text{ in degree }d-1\\
 X&4&X\text{ in degree }d-1
\end{array}
\]
We now impose $h^5-h=0$, so the parameter $a$ lies in $\mathbb F_5$.
Write $I_n$ and $Y_n$ for the simple quotients of $\Delta(\one,n)$ and
$\Delta(X,n)$, respectively.  At the four exceptional values, following
the singular subobject through the remaining homogeneous degree and using
\eqref{eq:sl1L3-Verma-characters} gives
\begin{equation}\label{eq:sl1L3-four-atypical-simples}
 \begin{aligned}
  I_0&=\one_0,&
  Y_4&=X_4\oplus\one_3,\\
  Y_3&=X_3\oplus X_2,&
  I_2&=\one_2\oplus X_1.
 \end{aligned}
\end{equation}
and the four nonsplit exact sequences
\begin{equation}\label{eq:sl1L3-atypical-Verma-sequences}
 \begin{gathered}
  0\longrightarrow Y_4\longrightarrow \Delta(\one,0)
    \longrightarrow I_0\longrightarrow0,\\
  0\longrightarrow Y_3\longrightarrow \Delta(X,4)
    \longrightarrow Y_4\longrightarrow0,\\
  0\longrightarrow I_2\longrightarrow \Delta(X,3)
    \longrightarrow Y_3\longrightarrow0,\\
  0\longrightarrow I_0\longrightarrow \Delta(\one,2)
    \longrightarrow I_2\longrightarrow0.
 \end{gathered}
\end{equation}

For the other six values, every Shapovalov map is an isomorphism.
Therefore \cref{lem:shapovalov-detects-simplicity} shows that
\[
 \begin{aligned}
 \Delta(\one,1)&=I_1,&
 \Delta(\one,3)&=I_3,&
 \Delta(\one,4)&=I_4,\\
 \Delta(X,0)&=Y_0,&
 \Delta(X,1)&=Y_1,&
 \Delta(X,2)&=Y_2
 \end{aligned}
\]
are simple.  More explicitly, as graded $\g(0)$-modules,
\[
 \begin{aligned}
 I_n&=\one_n\oplus X_{n-1}\oplus\one_{n-2}
 &&\text{for }n\in\{1,3,4\},\\
 Y_n&=X_n\oplus(\one_{n-1}\oplus X_{n-1})\oplus X_{n-2}
 &&\text{for }n\in\{0,1,2\}.
 \end{aligned}
\]
Note that, although $\beta_{\lambda_a,1}$ is nondegenerate for
$a\neq0$, the higher Shapovalov maps introduce the additional exceptional
values $2,3,4$.  For each $Z$, the complement of the exceptional set is
the nonempty open set predicted by
\cref{thm:generic-irreducibility-verp}.

\subsubsection{Projectives and blocks}

We now determine which simple modules are projective. 
Let
$  G=SL(\one\oplus X).$ 
Then \(\g=\Lie(G)\), and in this example
$  u(\g)=\operatorname{Dist}(G_1).$ 
Consequently,
\[
 \mathcal R
 \simeq
 \Rep_{\Ver_5^+}(G_1).
\]

To determine the projective objects, we compare with the Kac modules
studied in \cite[Section~9]{Utiralova}.  Since
$  L_3\cong L_2\otimes L_4,$ 
 her supergroup convention identifies
\(GL(L_1\oplus L_3)\) with \(GL(L_1\mid L_2)\) \footnote{Put
$  \widetilde G=GL(L_1\oplus L_3),$ 
$  G=SL(\one\oplus X)\subseteq\widetilde G.$ 
We choose the Borel of $\widetilde G$ whose positive and negative
nilradicals agree with the grading of $\g$.}. Let
\(K(Z,m)\) denote the Kac module with highest-weight object \(Z_m\),
where \(Z\in\{\one,X\}\) and \(m\in\mathbb Z\).  Its PBW decomposition is
\[
 K(Z,m)\cong S(\g(-1))\otimes Z_m.
\]
This is the same PBW construction as for \(\Delta(Z,m)\).
Upon restriction to \(G_1\), the element \(h\) satisfies \(h^5-h=0\).
Therefore it is enough to consider 
$\overline m\in\mathbb F_5$, and we obtain
\[
 \left.K(Z,m)\right|_{G_1}
 \cong
 \Delta(Z,\overline m).
\]
Thus passing from Kac modules to our baby Verma modules is equivalent 
to reducing the \(h\)-weights modulo \(5\).

In Utiralova's terminology, a highest weight is \emph{typical} if it
avoids the exceptional values at which an additional singular vector can
occur.  Her typicality criterion in 
\cite[Corollary~9.17 and Theorem~9.35]{Utiralova} specializes in this
example to
\[
 \begin{aligned}
 K(\one,m)\text{ is typical}
 &\iff \overline m\notin\{0,2\},\\
 K(X,m)\text{ is typical}
 &\iff \overline m\notin\{3,4\}.
 \end{aligned}
\]
These are exactly the six classes for which our Shapovalov
calculation showed that \(\Delta(Z,\overline m)\) is simple.
Moreover, an irreducible Kac module is projective.
Restriction from \(G\) to \(G_1\) preserves projective objects:
its right adjoint is coinduction, which is exact because \(G/G_1\) is
affine.  It follows that
\[
 \begin{gathered}
 P(I_1)=I_1,\qquad P(I_3)=I_3,\qquad P(I_4)=I_4,\\
 P(Y_0)=Y_0,\qquad P(Y_1)=Y_1,\qquad P(Y_2)=Y_2.
 \end{gathered}
\]
Since \(\mathcal R\) is a finite tensor category, projective objects are
also injective.  Hence each of these six simple-projective modules forms
a block by itself.
In particular, the module $Y_1$ is the adjoint module $\g$.  

The only nonsemisimple block is the principal block, with simple objects
$I_0,Y_4,Y_3,I_2$.  Write
$P I_n:=P(I_n)$ and $P Y_n:=P(Y_n)$.  The four exact sequences \eqref{eq:sl1L3-atypical-Verma-sequences},
together with BGG reciprocity \cite[Corollary~9.5]{Utiralova} and 
\cite[Theorem~9.46]{Utiralova}, determine the
Loewy series of its indecomposable projective covers:
\begin{equation}\label{eq:sl1L3-projectives}
 \begin{array}{c@{\qquad}c}
 PI_0&PY_4\\[1mm]
 \begin{matrix}
 I_0\\ I_2\oplus Y_4\\ I_0
 \end{matrix}
 &
 \begin{matrix}
 Y_4\\ I_0\oplus Y_3\\ Y_4
 \end{matrix}
 \end{array}
 \qquad
 \begin{array}{c@{\qquad}c}
 PY_3&PI_2\\[1mm]
 \begin{matrix}
 Y_3\\ Y_4\oplus I_2\\ Y_3
 \end{matrix}
 &
 \begin{matrix}
 I_2\\ Y_3\oplus I_0\\ I_2.
 \end{matrix}
 \end{array}
\end{equation}
In the cyclic order $I_0,Y_4,Y_3,I_2$, the Cartan matrix is
\begin{equation}\label{eq:sl1L3-Cartan}
 C=
 \begin{pmatrix}
 2&1&0&1\\
 1&2&1&0\\
 0&1&2&1\\
 1&0&1&2
 \end{pmatrix}.
\end{equation}

Notice that $\det C=0$. 
However, the Cartan matrix of a highest weight category with finitely many simple
objects has determinant $1$.  Consequently, $\mathcal R$ is not a highest
weight category in the sense of \cite{CPS}, despite its construction from
highest-weight baby Verma modules.  

\subsubsection{Tensor products and the Grothendieck ring}
The preceding calculations describe the simple objects, the block decomposition, and the Loewy series of the indecomposable projective covers of $\mathcal R$. We now record the tensor products of the simple modules.

 The decompositions of
the simple modules above give the dualities
\[
 I_2^*=Y_4,
 \qquad I_4^*=I_3,
 \qquad Y_0^*=Y_2,
\]
and the other simple modules are self-dual.  In the Grothendieck ring,
\begin{equation}\label{eq:sl1L3-projective-classes}
 \begin{aligned}
 [P I_0]&=2[I_0]+[I_2]+[Y_4],&
 [P I_2]&=[I_0]+2[I_2]+[Y_3],\\
 [P Y_4]&=[I_0]+2[Y_4]+[Y_3],&
 [P Y_3]&=[I_2]+[Y_4]+2[Y_3].
 \end{aligned}
\end{equation}
Using these characters, the projective covers above, and the fact that
tensoring preserves projectives gives the following decompositions.  Since
the tensor product is symmetric, the lower-triangular table determines all
products of simple modules.  In the row labeled $Z$, the entries are
$Z\otimes W$ for the preceding column labels $W$.

\vspace{10pt}
\begin{center}
\renewcommand{\arraystretch}{1.35}
\resizebox{\textwidth}{!}{$
\begin{array}{|c|c|c|c|c|c|c|c|c|c|}
\hline
 &I_1&I_2&I_3&I_4&Y_0&Y_1&Y_2&Y_3&Y_4
\\ \hline
I_1&P I_0\oplus Y_1\\ \cline{1-3}
I_2&I_3\oplus Y_2
 &I_4\oplus Y_3
\\ \cline{1-4}
I_3&I_4\oplus P I_2&P Y_4&I_1\oplus I_4\oplus Y_0\\ \cline{1-5}
I_4&I_3\oplus P Y_4&I_1\oplus Y_0&P I_0\oplus Y_1
 &I_3\oplus I_1\oplus Y_2
\\ \cline{1-6}
Y_0&\begin{matrix}P Y_4\oplus Y_1\oplus Y_0\end{matrix}
 &I_1\oplus Y_2\oplus Y_1
 &P I_2\oplus Y_2\oplus Y_1
 &P Y_3\oplus I_3\oplus Y_2
 &\begin{matrix}P Y_4\oplus P Y_3\oplus I_4\oplus I_3\oplus Y_0 \end{matrix}
\\ \cline{1-7}
Y_1&\begin{matrix}I_1\oplus Y_2\oplus Y_1\oplus Y_0\end{matrix}
 &P I_2\oplus Y_2
 &P Y_3\oplus I_3\oplus Y_2
 &P Y_3\oplus I_4\oplus Y_0
 &\begin{matrix}P Y_4\oplus2Y_0\oplus I_4\oplus I_1\oplus Y_1\end{matrix}
 &\begin{matrix}P I_0\oplus2Y_1\oplus Y_2\oplus I_1\oplus Y_0\end{matrix}
\\ \cline{1-8}
Y_2&\begin{matrix}P I_2\oplus Y_1\oplus Y_2\end{matrix}
 &I_3\oplus P Y_3
 &P Y_3\oplus I_4\oplus Y_0
 &P Y_4\oplus Y_1\oplus Y_0
 &\begin{matrix}P I_0\oplus2Y_1\oplus Y_2\oplus I_1\oplus Y_0\end{matrix}
 &\begin{matrix}P I_2\oplus2Y_2\oplus I_3\oplus I_1\oplus Y_1\end{matrix}
 &\begin{matrix} P I_2\oplus P Y_3\oplus I_4\oplus I_3\oplus Y_2\end{matrix}
\\ \cline{1-9}
Y_3&P Y_3
 &Y_0\oplus Y_4
 &Y_1\oplus Y_0
 &Y_2\oplus Y_1
 &\begin{matrix}P I_2\oplus I_3\oplus Y_2\end{matrix}
 &\begin{matrix}P Y_3\oplus I_4\oplus I_3\end{matrix}
 &\begin{matrix}P Y_4\oplus I_4\oplus Y_0\end{matrix}
 &\begin{matrix}Y_1\oplus I_0\oplus I_1\end{matrix}
\\ \cline{1-10}
Y_4&I_4\oplus Y_0
 &I_0\oplus Y_1
 &I_1\oplus Y_2
 &P I_2
 &I_4\oplus P Y_3
 &P Y_4\oplus Y_0
 &I_1\oplus Y_0\oplus Y_1
 &I_2\oplus Y_2
 &I_3\oplus Y_3
\\ \hline
\end{array}
$}
\end{center}
\vspace{10pt}

\section{Contragredient Lie algebras}
Throughout this section, $\C$ is assumed quasi-semisimple. We study contragredient Lie algebras, as introduced in  \cite{APS-contragredient}.
 After recalling their construction and basic properties, we
prove a criterion for simplicity for the derived
algebra. Lastly, we  develop a family of reflections which allow us to build Weyl groupoids and
root systems.

\subsection{Basic properties of contragredient Lie algebras}\label{subsec:contragredient-defs} In this section we set our notation. 
Given a Lie algebra $(\fh,\ttb_{\fh})$ in $\cC$, a \emph{contragredient datum} \cite{APS-contragredient} over $\fh$ is a pair $(\rho, \ttd)$ where:
\begin{enumerate}[leftmargin=*,label=\rm{(D\arabic*)}]
\item\label{item:Ver-contragredient-1} $\rho: \fh \otimes \ttV \to \ttV$ is an $\fh$-module in $\cC$, and  
\item \label{item:Ver-contragredient-2} $\ttd : \ttV \otimes \ttV^*\to \fh$ is a morphism in $\cC$, 
\end{enumerate}
satisfying the following compatibility:
\begin{align}\label{eq:contragredient-compatibility}
\ttb_{\fh}(\id_{\fh} \otimes \ttd) = \ttd(\rho \otimes \id_{\ttV^*})+\ttd(\id_{\ttV} \otimes \rho^{\vee})(c_{\fh, \ttV} \otimes \id_{\ttV^*}).
\end{align}
Here, $\rho^{\vee}: \fh \otimes \ttV^* \to \ttV^*$ is the dual module structure.

We will always assume $\fh$ has  \emph{enough modules}, i.e., for every simple subobject $S\subseteq \fh$ there exists an $\fh$-module $\lambda:\fh\otimes \ttW\to\ttW$ such that $\restr{\lambda}{S\otimes\ttW}\ne 0$. 

\medskip
For $X\in\cC$, let $\FLie(X)\subseteq T(X)$ be the Lie subalgebra generated by $X$. This is the free Lie algebra and has the PBW property. An action of $\fh$ on $X$ induces an embedding
\[
\FLie(X)\rtimes\fh\ \hookrightarrow\ T(X)\#U(\fh),
\]
where the right side is the associative smash product for the induced action on $T(X)$; its underlying object is $T(X)\otimes U(\fh)$. Since $\cC$ is quasi-semisimple, \cite[Proposition~7.7]{EtingofPBW} gives the PBW property for this semidirect product. Its quotients have the PBW property by \cite[Lemma~7.6(i)]{EtingofPBW}. This applies to the constructions below with $X=\ttV\oplus\ttV^*$.

Mimicking the construction of Kac-Moody Lie algebras, we denote by $\widetilde{\fg}(\rho, \ttd)$ the Lie algebra obtained as the quotient of $\FLie(\ttV\oplus \ttV^* )\rtimes_\rho \fh$ by the ideal generated by the image of 
\[(\ttb_{\ttV \otimes \ttV^*},-\ttd) : \ttV\otimes \ttV^* \longrightarrow \FLie(\ttV\oplus \ttV^* )\rtimes_\rho \fh,
\]
where $\ttb_{\ttV \otimes \ttV^*}$ denotes the restriction of the Lie bracket $\ttb$ of $\FLie(\ttV\oplus \ttV^* )$ to $\ttV\otimes   \ttV^*$.

Then \cite[Theorem 4.18]{APS-contragredient} states the following:
\begin{enumerate}[leftmargin=*,label=\rm{(\alph*)}]
\item 
We denote by $\widetilde{\fn}_+$ and $\widetilde{\fn}_-$ the subalgebras of $\widetilde{\fg}(\rho, \ttd)$ generated by $\ttV$ and $\ttV^*$, respectively. Then $\FLie(\ttV)\cong \widetilde{\fn}_+$ and $\FLie(\ttV^*)\cong\widetilde{\fn}_-$ as Lie algebras.
\item We have a triangular decomposition $\widetilde{\fg}=\widetilde{\fn}_-\oplus \fh \oplus\widetilde{\fn}_+$, and 
\begin{align}\label{eq:gtilde-verp-triangular}
\widetilde{\fn}_- &= \bigoplus_{k<0} \widetilde{\fg}_k, & \fh &= \widetilde{\fg}_0, & \widetilde{\fn}_+ &= \bigoplus_{k>0} \widetilde{\fg}_k. 
\end{align} 
\end{enumerate}
Moreover, there exists a unique maximal $\bZ$-homogeneous ideal $\ttm$ among all ideals of $\widetilde{\fg}$ trivially intersecting $\fh$, and $\ttm=\ttm_+\oplus \ttm_-$, where $\ttm_{\pm}=\ttm\cap\widetilde{\fn}_{\pm}$.

\smallskip

The next step is to take the quotient by $\ttm$. We expect for the quotient to keep $\ttV$, respectively $\ttV^*$, as the component of degree $1$, respectively $-1$. One way to ensure this property is to require that the datum $(\rho,\ttd)$ is \emph{reduced}, i.e. that for any simple subobjects $S\subseteq\ttV$ and $T\subseteq\ttV^*$, we have that $\restr{\ttd}{S\otimes\ttV^*}\ne 0$ and $\restr{\ttd}{\ttV\otimes T}\ne 0$. 

Assume then that the datum is reduced. By \cite[Lemma~4.28]{APS-contragredient}, $\ttI\cap\ttV=0$ and $\ttI\cap\ttV^*=0$ for every ideal $\ttI$ of $\widetilde{\fg}(\rho,\ttd)$ such that $\ttI\cap\fh=0$. In this context, the \emph{contragredient Lie algebra} associated to the triple $(\fh, \rho, \ttd)$ is the ind-Lie algebra $\fg(\rho, \ttd)\coloneqq \widetilde{\fg}(\rho, \ttd)/\ttm$.

The defining ideal of $\fg(\rho, \ttd)$ is homogeneous, so $\fg$ inherits a $\bZ$-grading
\begin{align}\label{eq:contragredient-Z-graded}
\fg(\rho, \ttd)&=\bigoplus_{k\in\bZ} \fg_k(\rho, \ttd),
\end{align}
where $\fg_k(\rho, \ttd)$ is the image of $\widetilde{\fg}_k(\rho, \ttd)$. We denote by $\fn_+$ and $\fn_-$ the subalgebras of $\fg(\rho, \ttd)$ generated by $\ttV$ and $\ttV^*$, respectively.

\begin{theorem}{\cite[Theorem 4.33]{APS-contragredient}}\label{thm:g-contragredient-verp}
\begin{enumerate}[leftmargin=*,label=\rm{(\alph*)}]
\item \label{item:g-contragredient-verp-triangular} 
We have $\fg=\fn_-\oplus \fh \oplus\fn_+$ as objects in $\indcat{\cC}$. Moreover,
\begin{align}\label{eq:g-contragredient-verp-triangular}
\fn_- &= \bigoplus_{k<0} \fg_k, & \fh &= \fg_0, & \fn_+ &= \bigoplus_{k>0} \fg_k. 
\end{align}
\item\label{item:g-contragredient-verp-trivial-bracket} If $Y$ is a $\bZ$-homogeneous subobject of $\fn_+$ (respectively, $\fn_-$) such that $[Y,\ttV^*]=0$ (respectively, $[Y,\ttV]=0$), then $Y=0$.
\end{enumerate}
\end{theorem}

\smallskip

Next we record a recognition criterion for the contragredient quotient; it will be used repeatedly in what follows.

\begin{proposition}{\cite[Proposition 4.39]{APS-contragredient}}\label{prop:contragredient-bilinear-form}
Let $\ttq=\bigoplus_{k\in\bZ}\ttq_k$ be a quotient of $\widetilde{\fg}(\rho,\ttd)$ by a $\bZ$-homogeneous ideal $\ttI$ satisfying $\ttI\cap\fh=0$. Assume that $\ttq$ has a degree zero non-degenerate symmetric invariant form $\ttB_{\ttq}:\ttq\otimes\ttq\to\one$. Then $\ttI=\ttm$, that is $\ttq\cong\fg(\rho,\ttd)$.
\end{proposition}

\medskip

Next we assume that $\fh$ is equipped with a symmetric non-degenerate invariant form. A \emph{symmetrizable contragredient datum} over $\fh$ is a pair $(\rho, \ttK)$ where
\begin{enumerate}[leftmargin=*,label=\rm{(SD\arabic*)}]
\item\label{item:symmetrizable-def-1} $\rho:\fh\otimes\ttV\to\ttV$ is an $\fh$-module in $\C$
such that $\rho(\fh\otimes S)\ne0$ for every simple subobject
$S\subseteq\ttV$, and $\pi\rho\ne0$ for every epimorphism
$\pi:\ttV\twoheadrightarrow T$ onto a simple object $T$ of $\C$.
\item\label{item:symmetrizable-def-2} $\ttK: \fh \otimes \fh \to \one$ is a symmetric non-degenerate invariant form.
\end{enumerate}

Equivalently, condition \ref{item:symmetrizable-def-1} says
$\ttV^{\fh}=0$ and
$\ima(\rho:\fh\otimes\ttV\to\ttV)=\ttV$, or, by duality,
\[
 \ttV^{\fh}=(\ttV^*)^{\fh}=0.
\]
Indeed, a nonzero invariant subobject has a simple subobject, while
a nonzero cokernel of $\rho$ has a simple quotient. The simple-quotient
formulation agrees with the original direct-summand formulation
when $\C$ is semisimple.

Given a symmetrizable datum $(\rho,\ttK)$, we define
$  \ttd_{\rho,\ttK}:\ttV\otimes\ttV^*\longrightarrow\fh$
to be the unique morphism satisfying
\begin{equation}\label{eq:symmetrizable-dmap-characterization}
 \ttK(\id_{\fh}\otimes\ttd_{\rho,\ttK})
 =
 \ev_{\ttV}(\rho\otimes\id_{\ttV^*})
 \quad\text{in }
 \Hom_{\cC}(\fh\otimes\ttV\otimes\ttV^*,\one).
\end{equation}
The existence and uniqueness of $\ttd_{\rho,\ttK}$ follow from the
non-degeneracy of $\ttK$.  As shown in
\cite{APS-contragredient}, the pair
$(\rho,\ttd_{\rho,\ttK})$ is a reduced contragredient datum.

\begin{remark}\label{rem:SD1-holds-ss}
Assumption \ref{item:symmetrizable-def-1} holds whenever $\ttV$ is a
finite direct sum of simple $\fh$-modules with nontrivial action.  Such a
module has no trivial simple $\fh$-submodule, which gives the first
condition.  For the second condition, let $\pi:\ttV\twoheadrightarrow T$
be any epimorphism onto a simple object of $\C$. If
$\pi\rho=0$, then $\pi$ is a morphism of $\fh$-modules from $\ttV$ to
$T$ equipped with the trivial action.  This is impossible because
$\ttV$ is a direct sum of nontrivial simple $\fh$-modules.
\end{remark}

We fix a decomposition of $\rho: \fh \otimes \ttV \to \ttV$ as a sum of $\fh$-submodules, say $\ttV=\mathop{\oplus}\limits_{1\leq i \leq r} \ttV_i$. For the associated reduced datum $(\rho, \ttd)$, we have that
\begin{align}\label{eq:symmetrizable-datum-ortogonal-summands}
\ttd(\ttV_i\otimes \ttV_j^*)&=0 \text{ for all }i\ne j, & \text{so }\ttd(\ttV_i\otimes \ttV_i^*) &\ne 0 \text{  for all }1\le i\le r.
\end{align}
Let $Q=\bigoplus_{j=1}^{r}\bZ\alpha_j$ be the grading lattice, and $Q^\pm:=\mathop{\oplus}\limits_{1\leq i \leq r} \pm\bN_0\alpha_i$. 

\begin{theorem}{\cite[Theorems 4.50 and 4.54]{APS-contragredient}}\label{thm:contragredient-symmetrizable-grading}
The Lie algebra $\fg:= \fg(\rho,\ttd)$ is $Q$-graded, with
\begin{align}\label{eq:contragradient-symmetrizable-grading-def} 
\fg&=\bigoplus_{\alpha \in Q} \fg_\alpha, &\deg \fh&=0, &\deg \ttV_i&=\alpha_i, &\deg \ttV_i^*&=-\alpha_i, & 1&\leq i\leq r.
\end{align}
Moreover, $\fg_0=\fh$, $\fn_{\pm}=\bigoplus_{\alpha \in Q^{\pm},\, \alpha \neq 0} \fg_\alpha$. 

In addition, $\fg$ has an invariant symmetric form $\ttB:\fg\otimes\fg\to\one$ such that $\restr{\ttB}{\fh\otimes\fh}=\ttK$ and $\restr{\ttB}{\fg_\alpha\otimes\fg_\beta}=0$ if $\alpha+\beta\ne0$; thus, $\restr{\ttB}{\fg_\alpha\otimes\fg_{-\alpha}}$ is non-degenerate for all $\alpha \in Q$.
\end{theorem}

\subsection{Splitting} We now study how the contragredient construction behaves under a
decomposition of its generating module.  This is the categorical analogue
of separating a subset of simple roots in the classical theory, and
parallels the corresponding construction for Nichols algebras.

Assume that $\ttV=\ttU\oplus\ttU'$ as $\fh$-modules, and identify $\ttV^*=\ttU^*\oplus(\ttU')^*$. Consider the induced $\bZ^2$-grading in  $\fg$ given by
\begin{align*}
\deg\ttU&=(1,0), & \deg\ttU'&=(0,1), & \deg\ttU^*&=(-1,0), & \deg(\ttU')^*&=(0,-1).
\end{align*} 
We will also use the $\bZ$-grading induced by fixing the second coordinate:
\begin{equation}\label{eq:splitting-coarse-grading}
\fg_n\coloneqq\bigoplus_{m\in\bZ}\fg_{m,n}.
\end{equation}
Let $\ttY\subseteq\fg$ be the Lie subalgebra generated by $\ttU$, $\fh$, and $\ttU^*$. Let $\rho_{\ttU}:\fh\otimes\ttU\to\ttU$ be the restriction of the action and set
\[
\ttd_{\ttU}\coloneqq
\restr{\ttd}{\ttU\otimes\ttU^*}
=\ttd_{\rho_{\ttU},\ttK}:\ttU\otimes\ttU^*\longrightarrow\fh
\subseteq\ttY,
\]
where the equality follows from
\eqref{eq:symmetrizable-dmap-characterization}.

\begin{lemma}\label{lem:Y-contragredient}
We have $\ttY=\fg_0$, and the inclusion $\ttY\hookrightarrow\fg$ induces an isomorphism
\begin{align*}
\ttY\cong\fg(\rho_{\ttU},\ttd_{\ttU}).
\end{align*}
\end{lemma}

\begin{proof}
As $\ttU=\fg_{1,0}$, $\fh=\fg_{0,0}$ and $\ttU^*=\fg_{-1,0}$, we have that $\ttY\subseteq\fg_0$. 
On the other hand, $\fg_{1,0}=\ttU$ and for $m>1$, the equality
$\fg_{m,0}=\sum_{k=1}^{m-1}[\fg_{k,0},\fg_{m-k,0}]$
and induction show that $\fg_{m,0}$ belongs to the Lie subalgebra generated by $\ttU$.
The case $m<0$ is analogous with $\ttU^*$ in place of $\ttU$, while $\fg_{0,0}=\fh$. Thus $\fg_0=\bigoplus_{m\in\bZ}\fg_{m,0}\subseteq\ttY$.

Next we show that the pair $(\rho_{\ttU},\ttK)$ is symmetrizable: \ref{item:symmetrizable-def-1} follows since $\ttU$ is a direct summand of $\ttV$ while \ref{item:symmetrizable-def-2} holds by hypothesis. 

Now, there is a canonical Lie algebra graded map $\Phi_{\ttU}:\widetilde{\fg}(\rho_{\ttU},\ttd_{\ttU})\to\ttY$, which is the identity on $\fh$. Notice that $\Phi_{\ttU}$ is surjective since $\ttU$, $\fh$, and $\ttU^*$ belong to its image, and its kernel is homogeneous and intersects $\fh$ trivially. The restriction of $\ttB$ to $\ttY$ is invariant and symmetric, and $\restr{\ttB}{\fg_{m,0}\otimes\fg_{-m,0}}$ is non-degenerate for every $m$. By Proposition~\ref{prop:contragredient-bilinear-form}, $\ker\Phi_{\ttU}$ is the maximal defining ideal, so $\Phi_{\ttU}$ induces an isomorphism $\fg(\rho_{\ttU},\ttd_{\ttU})\simeq\ttY$.
\end{proof}

\begin{lemma}\label{lem:W-splitting}
\begin{enumerate}[leftmargin=*,label=\rm{(\alph*)}]
\item\label{item:W-decomposes} The $\ttY$-submodule $\ttW$ of  $\fg$ generated by $\ttU'$ coincides with $\fg_1=\oplus_{m\geq 0} \fg_{m,1}$.

\item \label{item:W*-decomposes} The $\ttY$-submodule of  $\fg$ generated by $(\ttU')^*$ coincides with $\fg_{-1}=\oplus_{m\geq 0} \fg_{-m,-1}$. If $\ttW\in\cC$, the form $\ttB$ induces a $\ttY$-module isomorphism $\ttW^*\simeq\fg_{-1}$. 
\end{enumerate} 
\end{lemma}

\begin{proof}
\ref{item:W-decomposes} By the triangular decomposition of $\fg$, we know that $\oplus_{m\geq 0} \fg_{m,1}=\oplus_{m\in \bZ} \fg_{m,1}$. Thus $\oplus_{m\geq 0}\fg_{m,1}$ is a $\ttY$-submodule that contains $\ttU'$, which proves the inclusion $\ttW\subseteq\oplus_{m\geq 0} \fg_{m,1}$. 

Also, by antisymmetry and the Jacobi identity, one can prove recursively that 
\[  \fg_{m,1}\subseteq [\fg_{1,0}, \cdots [\fg_{1,0}, \fg_{0,1}]]  \subseteq \ttW,\]
which proves the converse inclusion.

The first statement in \ref{item:W*-decomposes} follows by the same argument, and the second one since $\fg_\alpha$ and $\fg_\beta$ are $\ttB$-orthogonal unless $\alpha=-\beta$ for $\alpha$ and $\beta$ in $\bZ^2$, and $\ttB$ is non-degenerate.
\end{proof}

Assume from now on in this subsection that $\ttY,\ttW\in\cC$ and that
$\ttY$ has enough modules. Put
\[
\ttK_{\ttY}\coloneqq\restr{\ttB}{\ttY\otimes\ttY},
\]
and let $\eta:\ttY\otimes\ttW\to\ttW$ be the restricted adjoint action.

\begin{lemma}\label{lem:splitting-symmetrizable-datum}
The pair $(\eta,\ttK_{\ttY})$ is a symmetrizable datum over $\ttY$. Under
the identification $\ttW^*\cong\fg_{-1}$ from
Lemma~\ref{lem:W-splitting}, its associated morphism is the bracket:
\[
\ttd_{\eta,\ttK_{\ttY}}
=\restr{[\, , \, ]}{\ttW\otimes\ttW^*}:\ttW\otimes\ttW^*\longrightarrow\ttY.
\]
\end{lemma}

\begin{proof}
The form $\ttK_{\ttY}$ is symmetric and invariant. It is non-degenerate because it is the direct sum of the non-degenerate pairings $\fg_{m,0}\otimes\fg_{-m,0}\to\one$.

Let $S\subseteq\ttW$ be simple and suppose that $\eta(\ttY\otimes S)=0$. Since $[S,\fg_{-1}] \subseteq \ttY$ and that $\ttB$ is symmetric and invariant, we have that
\[
\ttK_{\ttY}\bigl( [S, \fg_{-1}], \ttY\bigr) =\ttB\bigl( \eta(\ttY\otimes S), \fg_{-1}\bigr)=0.
\]
Thus $[S, \fg_{-1}]=0$ since $\ttK_{\ttY}$ is non-degenerate. We choose $m$ such that the composite of $S\hookrightarrow\ttW$ with the projection to $\fg_{m,1}$ is non-zero, and we denote its image by $S_m$. Then $S_m$ is simple. Since $\ttU^* \subset \ttY$ and $(\ttU')^* =\fg_{0,-1}\subset \fg_{-1}$, we have that
\begin{align*}
[S_m,\ttU^*]&= -\eta(\ttU^*\otimes S_m) = 0, 
&  [S_m, (\ttU')^*]&=0.
\end{align*}
Hence $[S_m,\ttV^*]=0$, which contradicts Theorem~\ref{thm:g-contragredient-verp}~\ref{item:g-contragredient-verp-trivial-bracket}. Therefore $\eta(\ttY\otimes S)\ne0$.

The same argument proves that  $\pi\eta\neq 0$, where $\pi:\ttW\twoheadrightarrow T$ is an epimorphism onto a simple object.
Thus \ref{item:symmetrizable-def-1} holds, and $(\eta,\ttK_{\ttY})$ is symmetrizable.

Finally, for $y\in\ttY$, $w\in\ttW$, and $\varphi\in\ttW^*\cong\fg_{-1}$, since $\ttB$ is invariant, 
\[
\ttK_{\ttY}(y,[w,\varphi]) =\ttB([y,w],\varphi) =\ev_{\ttW}(\eta(y,w),\varphi).
\]
By \eqref{eq:symmetrizable-dmap-characterization}, this identity uniquely characterizes $\ttd_{\eta,\ttK_{\ttY}}$, so $\ttd_{\eta,\ttK_{\ttY}}=\restr{[\, , \, ]}{\ttW\otimes\ttW^*}$.
\end{proof}

We now arrive at our main result in this subsection.

\begin{theorem}\label{thm:splitting-contragredient}
There is a canonical isomorphism of $\bZ$-graded Lie algebras
\[
\fg(\rho,\ttd)\cong \fg(\eta,\ttd_{\eta,\ttK_{\ttY}}).
\]
\end{theorem}

\begin{proof}
The inclusions of $\ttY=\fg_0$, $\ttW=\fg_1$, and
$\ttW^*=\fg_{-1}$ into $\fg$, together with
Lemma~\ref{lem:splitting-symmetrizable-datum}, define a graded map
\[
\widetilde{\Phi}:
\widetilde{\fg}(\eta,\ttd_{\eta,\ttK_{\ttY}})\longrightarrow\fg.
\]
Its image contains $\ttY$, hence $\fh,\ttU,\ttU^*$. It also contains $\ttW$ and $\ttW^*$, hence $\ttU'$ and $(\ttU')^*$. Thus $\widetilde{\Phi}$ is surjective. Its kernel is homogeneous and intersects $\ttY$
trivially. Notice that the restriction of $\ttB$ to $\fg_n\otimes\fg_{-n}$ is the direct sum of the non-degenerate pairings
\[
\fg_{m,n}\otimes\fg_{-m,-n}\longrightarrow\one.
\]
Applying Proposition~\ref{prop:contragredient-bilinear-form} to $\ttY$, we get that $\ker\tilde\Phi$ is the maximal defining ideal. Hence $\widetilde{\Phi}$ induces an isomorphism $\Phi:\fg(\eta,\ttd_{\eta,\ttK_{\ttY}})\overset{\sim}{\longrightarrow}\fg$.
\end{proof}

\subsection{On the simplicity of contragredient Lie algebras}

As before, $(\rho, \ttK)$ denotes a symmetrizable contragredient datum over $\fh$. Set $\fg\coloneqq \fg(\rho,\ttd)$.  We now introduce a criterion for the simplicity of $\fg$.

Even when finite-dimensional, a contragredient Lie algebra may have a
non-zero center and need not be perfect, as already occurs for
contragredient Lie algebras and Lie superalgebras in positive
characteristic.  Since one of our objectives is to obtain simple Lie
algebras, we are therefore led to consider the derived algebra and its
quotient by its center. 

Let $\fh'_{\ttd}\subseteq\fh$ be the Lie
subalgebra generated by $\fh'$ and $\ima(\ttd)$.  As observed in
\cite{APS-contragredient}, if $\rho$ is surjective, then $\fg'$ is
generated by $\ttV^*,\fh'_{\ttd},\ttV$, and
$  \fg'=\fn_-\oplus\fh'_{\ttd}\oplus\fn_+.$

\begin{lemma}\label{lem:center-contragredient} Assume $\rho$ is surjective. With the notation above, 
let $\fc:=\fc(\fg)$ and $\fc':=\fc(\fg')$. Then
\begin{align*}
\fc&=\fc(\fh)\cap\Ann\ttV, & \fc'&=\fc(\fh'_{\ttd})\cap\Ann\ttV.
\end{align*} 
In particular, if $\fh$ is abelian, then 
\begin{align*}
\fc&=\Ann\ttV, & \fc'&=\ima(\ttd)\cap\Ann\ttV.
\end{align*} 
\end{lemma}

\begin{proof}
We first show that $\fc$ is a graded subobject of $\fg$. Let $p_k:\fg\to\fg_k$ denote the homogeneous projection. Since the bracket is homogeneous, for every $\ell\in\mathbb Z$ we have
\[
 [p_k(\fc),\fg_\ell] = p_{k+\ell}\left([\fc,\fg_\ell] \right) = 0.
\]
Hence $p_k(\fc)\subseteq\fc$. Therefore $\fc=\bigoplus_{k\in\mathbb Z}\fc_k$, where $\fc_k:=\fc\cap\fg_k$.

Suppose that $k>0$. Then $\fc_k$ is a homogeneous subobject of $\bigoplus_{r>0}\fg_r$ and $[\fc_k,\ttV^*]=0$. Thus $\fc_k=0$ by \Cref{thm:g-contragredient-verp} \ref{item:g-contragredient-verp-trivial-bracket}.
Similarly, if $k<0$, the same result gives $\fc_k=0$. Consequently,
\[
 \fc=\fc_0\subseteq\fg_0=\fh.
\]
Since $\fc$ is central in $\fg$, it is central in $\fh$ and acts
trivially on $\ttV$ and $\ttV^*$. Therefore
\[
\fc \subseteq \fc(\fh)\cap\Ann\ttV\cap\Ann\ttV^*  = \fc(\fh)\cap\Ann\ttV.
\]

On the other hand, $\fc(\fh)\cap\Ann\ttV$ is central in $\fh$ and acts trivially on both $\ttV$ and $\ttV^*$. Since $\fg_{-1}\oplus\fg_0\oplus\fg_1=\ttV^*\oplus\fh\oplus\ttV$ generate $\fg$ as a Lie algebra, it follows that $\fc(\fg)=\fc(\fh)\cap\Ann\ttV$.

As $\fg'$ is generated by $\ttV^*$, $\fh'_d$, and $\ttV$,  the same argument shows that $\fc'=\fc(\fh'_d)\cap\Ann\ttV$.

Finally, if $\fh$ is abelian, then $\fc(\fh)=\fh$ and $\fh'_{\ttd}=\operatorname{im}\ttd$, so the last statement is immediate.
\end{proof}

From now on we assume that $\ttV$ is semisimple as $\fh$-module, so by Remark \ref{rem:SD1-holds-ss} condition \ref{item:symmetrizable-def-1} holds. To draw a parallel with the classical theory of contragredient Lie algebras, we fix a decomposition $\ttV=\mathop{\oplus}_{i=1}^r \tte_i$ as a sum of simple $\fh$-submodules, and set $\ttf_i=\tte_i^*$, so $\ttV^*=\mathop{\oplus}_{i=1}^r \ttf_i$, and $\tth_i=[\tte_i,\ttf_i]$, so $\ima \ttd=\sum_{i=1}^r \tth_i$. 

We consider a quiver $Q_{\rho,\ttd}$ with vertices $1,\cdots,r$ and arrows $i\to j$ for each pair of vertices $i,j$ such that $\rho(\tth_i\otimes\tte_j)\ne0$. 
For each arrow $i\to j$ we have that $\rho(\tth_i\otimes\tte_j)=\tte_j$ since $\rho(\tth_i\otimes\tte_j)$ is a non-zero $\fh$-submodule of $\tte_j$.

\begin{definition}
We say that $(\rho,\ttd)$ is \emph{indecomposable} if $Q_{\rho,\ttd}$ is connected; that is, if for each pair $i,j\in\{1,\cdots,r\}$, there exist $i_0=i$, $i_1$, $\cdots, i_k=j$ such that $\rho(\tth_{i_{t-1}}\otimes\tte_{i_t})\ne0$ for all $t=1,\cdots,k$.
\end{definition}

\begin{lemma}\label{lem:ideals-contragredient-indec}
Let $(\rho,\ttd)$ be indecomposable such that $\fg\in\cC$, the quiver has at least one arrow and $\fg_0'$ is generated by the subobjects $\tth_i$, $1\leq i\leq r$, and each $\tte_i$ is simple after restriction to $\fg'_0$.
Then every ideal $\ttI$ of $\fg'$ satisfies either
$\ttI\subseteq\Ann\ttV\cap\fg'_0$ or $\ttI=\fg'$.
\end{lemma}

\begin{proof}
Set $\mathfrak a_0:=\fg_0'$ and $\mathfrak z=\Ann\ttV\cap\mathfrak a_0$. As $\ttd$ is equivariant, we have that $\ima\ttd=\sum_i\tth_i$ is an ideal of $\fh$. Since $\mathfrak a_0$ is generated by the $\tth_i$, it follows that $\mathfrak a_0=\ima\ttd$. The subobject $\mathfrak z$ acts trivially
on both $\ttV$ and $\ttV^*$. It therefore commutes with
$[\ttV,\ttV^*]=\ima\ttd=\mathfrak a_0$, and hence is central in $\fg'$. Set $\overline{\fg}=\fg'/\mathfrak z$. The nonzero-degree components are unchanged, while the action of $\overline{\fg}_0=\mathfrak a_0/\mathfrak z$ on $\ttV$ is faithful.

We first prove that every nonzero ideal $J$ of $\overline{\fg}$ is the whole of $\overline{\fg}$. 
Let $p_k:\overline{\fg}\to\overline{\fg}_k$ be the homogeneous
projections, and choose the least integer $k$ such that $p_k(J)\ne0$.
Since $\overline{\fg}$ has finite graded support,
\[
 J\subseteq\bigoplus_{n\ge k}\overline{\fg}_n.
\]
Set $X=J$ and $X_k=p_k(J)$. The projection is
$\mathfrak a_0$-linear, so $X_k$ is automatically an
$\mathfrak a_0$-submodule. In schematic notation, we write $x=x_k+x_{>k}$. We consider three cases.

\begin{itemize}[leftmargin=*]
\renewcommand{\labelitemi}{$\circ$}

\item If $k<0$, repeated bracketing with suitable summands $\tte_i$,
using
\Cref{thm:g-contragredient-verp}%
\ref{item:g-contragredient-verp-trivial-bracket}, gives a subobject
\[
W\subseteq
J\cap\bigoplus_{n\geq-1}\overline{\fg}_n
\]
whose lowest homogeneous image
$W_{-1}\subseteq\ttV^*$ is nonzero. Some projection
$W_{-1}\to\ttf_i$ is nonzero and hence surjective, because $\ttf_i$
is simple as an $\mathfrak a_0$-module. Choose an arrow
$i\to\ell$. Since $[\tte_j,\ttf_i]=0$ for $j\ne i$, the subobject
\[
Y^{(\ell)}
:=
[\tte_\ell,[\tte_i,W]]
\subseteq
J\cap\bigoplus_{n\geq1}\overline{\fg}_n
\]
has lowest homogeneous image $\tte_\ell$ in degree $1$.

\item If $k>0$, repeated bracketing with suitable summands $\ttf_i$
gives a subobject
\[
W\subseteq
J\cap\bigoplus_{n\geq1}\overline{\fg}_n
\]
with nonzero lowest homogeneous image $W_1\subseteq\ttV$. Some
projection $W_1\to\tte_i$ is surjective. Choosing an arrow
$i\to\ell$, the subobject
\[
Y^{(\ell)}
:=
[\tte_\ell,[\ttf_i,W]]
\subseteq
J\cap\bigoplus_{n\geq1}\overline{\fg}_n
\]
has lowest homogeneous image $\tte_\ell$.

\item If $k=0$, then $X_0$ acts nontrivially on $\ttV$, since the
action of $\overline{\fg}_0$ on $\ttV$ is faithful. Thus
$[X_0,\tte_\ell]\ne0$ for some $\ell$. This image is a nonzero
$\mathfrak a_0$-submodule of $\tte_\ell$, so the simplicity
hypothesis gives
 $[X_0,\tte_\ell]=\tte_\ell.$ 
Therefore
\[
Y^{(\ell)}
:=
[X,\tte_\ell]
\subseteq
J\cap\bigoplus_{n\geq1}\overline{\fg}_n
\]
has lowest homogeneous image $\tte_\ell$.

\end{itemize}

Thus, in every case we have constructed some
$Y^{(\ell)}\subseteq J$ whose lowest term is $\tte_\ell$. If
$\ell\to\ell'$ is an arrow, then
\[
Y^{(\ell')}
:=
[\tte_{\ell'},[\ttf_\ell,Y^{(\ell)}]]
\subseteq
J\cap\bigoplus_{n\geq1}\overline{\fg}_n
\]
has lowest homogeneous image $\tte_{\ell'}$. Following directed paths
gives such a subobject $Y^{(i)}$ for every vertex $i$.

Since $\fg\in\cC$, there is a largest $N>0$ for which
\(
\fg_N=\overline{\fg}_N\ne0.
\)
The algebra $\fn_+$ is generated by $\ttV$, so for some
$i_1,\ldots,i_N$ the image
\(
[\tte_{i_1},\ldots,
 [\tte_{i_{N-1}},\tte_{i_N}]]
\)
is nonzero. Replacing the last factor by $Y^{(i_N)}$ gives
\[
\begin{aligned}
[\tte_{i_1},\ldots,
 [\tte_{i_{N-1}},Y^{(i_N)}]]
&=
[\tte_{i_1},\ldots,
 [\tte_{i_{N-1}},\tte_{i_N}]] \subseteq J\cap\overline{\fg}_N.
\end{aligned}
\]
Indeed, every term of $Y^{(i_N)}$ above degree $1$ would produce a
term above degree $N$, and hence vanishes. We have therefore obtained
an actual nonzero homogeneous subobject of $J$.

Repeated bracketing with suitable $\ttf_j$, using
\Cref{thm:g-contragredient-verp}%
\ref{item:g-contragredient-verp-trivial-bracket}, now gives a
nonzero subobject of $J\cap\ttV$. This intersection is an
$\mathfrak a_0$-submodule. Some projection
\(
J\cap\ttV\longrightarrow\tte_i
\)
is nonzero, hence surjective by the simplicity hypothesis. Since
$[\tte_j,\ttf_i]=0$ for $j\ne i$, bracketing with $\ttf_i$ shows that
\(
\overline{\tth}_i\subseteq J.
\)

For an arrow $i\to j$, we have
\(
[\overline{\tth}_i,\tte_j]=\tte_j.
\)
Following directed paths, and bracketing each new $\tte_j$ with
$\ttf_j$, gives all $\tte_j$ and all $\overline{\tth}_j$ in $J$.
If $i\to j$, duality gives
\(
[\overline{\tth}_i,\ttf_j]=\ttf_j
\)
as an equality of images. Hence all $\ttf_j$ lie in $J$. These
subobjects generate $\overline{\fg}$, and therefore
\(
J=\overline{\fg}.
\)

Finally, let $\ttI$ be an ideal of $\fg'$ not contained in
$\mathfrak z$. Its image in $\overline{\fg}$ is nonzero, so the result
just proved gives
\(
\ttI+\mathfrak z=\fg'.
\)
It remains to remove the possible central summand $\mathfrak z$.
Every vertex has an incoming arrow. If $i\to j$, then the images of
$[\tth_i,\tte_j]$ and $[\tth_i,\ttf_j]$ are respectively $\tte_j$
and $\ttf_j$. Therefore every $\tte_j$ and $\ttf_j$ belongs to
$[\fg',\fg']$, and their brackets generate
$\mathfrak a_0=\ima\ttd$. Hence $\fg'$ is perfect. Since
$\mathfrak z$ is central,
\[
\fg'
=
[\fg',\fg']
=
[\ttI+\mathfrak z,\ttI+\mathfrak z]
\subseteq\ttI.
\]
Thus $\ttI=\fg'$, as required.
\end{proof}

Now we are ready to prove our simplicity criterion.

\begin{theorem}\label{thm:contragredient-indec-simple}
Let $(\rho,\ttd)$ be indecomposable such that $\fg\in\cC$ and $\fg_0'$ is generated by $\tth_i$, $1\leq i \leq r$, as Lie algebra, and its quiver has at least one arrow.
Assume in addition that each $\tte_i$ is simple after restriction to $\fg'_0$.
Then $\fg'/(\Ann\ttV\cap \fg_0')$ is a simple Lie algebra.
\end{theorem}

\begin{proof}

Set $\mathfrak z=\Ann\ttV\cap\fg'_0$. As $\ttd$ is $\fh$-linear, $\ima\ttd=\sum_i\tth_i$ is an ideal of $\fh$, and by our hypothesis, $\fg'_0=\ima\ttd$. The ideal $\Ann\ttV$ centralizes $\ima\ttd$, so Lemma~\ref{lem:center-contragredient} gives $\mathfrak z=\fc(\fg')$. The preimage in $\fg'$ of any ideal of $\fg'/\mathfrak z$ contains $\mathfrak z$; by Lemma~\ref{lem:ideals-contragredient-indec}, it is either $\mathfrak z$ or $\fg'$. Hence the quotient has no nonzero proper ideals. Finally, an arrow $i\to j$ gives
$[\tth_i,\tte_j]=\tte_j\ne0$, and this bracket remains nonzero after quotienting by $\mathfrak z\subseteq\fg'_0$. Thus the quotient is nonabelian and is simple.
\end{proof}

\begin{remark}
The simple Lie algebra   $\fg'/(\Ann\ttV\cap \fg_0')$ is $\bZ$-graded with $\fg_0'/\Ann\ttV\cap \fg_0'$ as homogeneous component of degree $0$ and the same components of degree $n\ne 0$ as $\fg$. This follows since $\Ann\ttV\cap \fg_0'$ is an ideal contained in $\fg_0'$.
\end{remark}

\subsection{Reflections} We now use our previous results to construct a family of reflections. 
As in the previous subsection, fix $\ttV=\bigoplus_{j=1}^{r}\tte_j$ and  $\ttV^*=\bigoplus_{j=1}^{r}\ttf_j$, where $\ttf_j=\tte_j^*$, $\tth_i = [\tte_i,\ttf_i]$, and each $\tte_j$ is a simple $\fh$-module. Let $Q=\bigoplus_{j=1}^{r}\bZ\alpha_j$ be the grading lattice, and $\he:Q\to\bZ$ the usual height function, $\he \left(\sum_j n_j\alpha_j\right):=\sum_j n_j$.

Let $\ttY_i$ be the Lie subalgebra generated by $\fh,\tte_i,\ttf_i$. By \Cref{lem:Y-contragredient}, $\ttY_i\simeq \fg(\rho_{\tte_i},\ttd_{\tte_i})$.
Write $\fn_i^{\pm}\coloneqq\ttY_i\cap\fn_{\pm}$.

We want to define a Cartan matrix attached to $\ttV$, for which we first need to check some finiteness obstructions. Fix $i$, and set $c_{ii}^{\ttV}:=2$. For $i\ne j$ define
\begin{equation}\label{eq:defn-Nij}
 N_{ij}^{\ttV}:= \sup\{n\in\bN_0:\fg_{\alpha_j+n\alpha_i}\ne0\} \in\bN_0\cup\{\infty\}.
\end{equation}
When $N_{ij}^{\ttV}<\infty$, let
\begin{align}\label{eq:defn-cij}
 c_{ij}^{\ttV} & :=-N_{ij}^{\ttV}.
\end{align}
For $n\geq0$, write
\begin{align*}
P_{ij,n}&\coloneqq\fg_{\alpha_j+n\alpha_i}, & P^-_{ij,n}&\coloneqq\fg_{-\alpha_j-n\alpha_i}.
\end{align*}

\begin{lemma}\label{lem:galphaj+nalphai-simple}
Assume $N=N_{ij}^{\ttV}<\infty$.
\begin{enumerate}[leftmargin=*,label=\rm{(\alph*)}]
\item\label{item:interval} For every $n\geq0$, $P_{ij,n}=(\ad\tte_i)^n(\tte_j)$, $P^-_{ij,n}=(\ad\ttf_i)^n(\ttf_j)$. Consequently,
\begin{align*}
\{n\in\bN_0:P_{ij,n}\ne0\}=\{0,1,\ldots,N\}.
\end{align*}
\item\label{item:galphaj+nalphai-simple} The object $M_{ij}\coloneqq\bigoplus_{n=0}^{N}P_{ij,n}$ is a simple $\ttY_i$-module. Its top component $P_{ij,N}$ is a simple $\fh$-module.
\end{enumerate}
\end{lemma}

\begin{proof}
\ref{item:interval} The fact that $P_{ij,n}=(\ad\tte_i)^n(\tte_j)$ and $P^-_{ij,n}=(\ad\ttf_i)^n(\ttf_j)$ follows arguing as in Lemma \ref{lem:W-splitting} \ref{item:W-decomposes}. In particular, $P_{ij,n}=0$ implies $P_{ij,n+1}=0$. Since $P_{ij,0}=\tte_j\ne0$ and $N$ is the largest non-zero index, the support is exactly $\{0,\ldots,N\}$.

\noindent \ref{item:galphaj+nalphai-simple}  Note that the object $M_{ij}$ is stable under $\fh$ and $\tte_i$. It is also stable under $\ttf_i$, since for $n=0$, $[P_{ij,0},\ttf_i]=[\tte_j,\ttf_i]=0$, while for $n>0$, $[P_{ij,n},\ttf_i]\subseteq P_{ij,n-1}$. Hence $M_{ij}$ is a $\ttY_i$-submodule.

Now we check that $M_{ij}$ is simple. First let $0\ne M\subseteq M_{ij}$ be a graded $\ttY_i$-submodule, and choose the least $n$ with $M_n\coloneqq M\cap P_{ij,n}\ne0$. If $n>0$, then $[\ttf_i,M_n]=0$. Since
$P^-_{ij,n}=[\ttf_i,P^-_{ij,n-1}]$, invariance gives
\[
 \ttB(M_n, P^-_{ij,n}) =\ttB([M_n,\ttf_i], P^-_{ij,n-1})=0,
\]
which is a contradiction since the pairing $P_{ij,n}\otimes P^-_{ij,n}\to\one$ is non degenerate. Thus $n=0$. The object $M_0$ is a non-zero $\fh$-submodule of $\tte_j$, so $M_0=\tte_j$, and $M=M_{ij}$.

For an arbitrary non-zero $\ttY_i$-submodule $M$, use the filtration on $M$ induced by 
\[
 F_{n}M=M \cap \left(\bigoplus_{0\leq r\leq n}P_{ij,r}\right),
\]
and the degree filtration on $\ttY_i$. The associated graded module $\gr M$ is a non-zero graded $\ttY_i$-submodule and hence $\gr M = M_{ij}$. Since the passage to the associated graded module preserves length, $M=M_{ij}$.

Finally, let $0\ne S\subseteq P_{ij,N}$ be an $\fh$-submodule. As $\fn_i^{+}$ annihilates $S$, Lemma~\ref{lem:submod-gen-from-V+} shows that the $\ttY_i$-submodule generated by $S$ is $U(\fn_i^-)\cdot S$, so $U(\fn_i^-)\cdot S=M_{ij}$. Its component of degree $N$ is exactly $S$, hence $S=P_{ij,N}$. Thus $S$ is simple.
\end{proof}

Following the terminology used for Nichols algebras \cite{HS}, we say that an ordered module $\ttV=(\tte_1,\ldots,\tte_r)$ is \emph{$i$-finite} if $N_{ij}^{\ttV}<\infty$ for every $j\ne i$. For an $i$-finite module $\ttV$, set
\begin{align*}
 R_i(\ttV)&:=(\tte_1^{(i)},\ldots,\tte_r^{(i)}), & &\text{where} & \tte_i^{(i)}&=\ttf_i, & \tte_j^{(i)}&=P_{ij,N_{ij}^{\ttV}}\quad(j\ne i).
\end{align*}
Let $\rho^{(i)}$ denote the $\fh$-action on $R_i(\ttV)$. Define also
\begin{align*}
s_i^{\ttV}: & \, \bZ^{r}\to \bZ^{r}, & s_i^{\ttV}(\alpha_j)&:=\alpha_j-c_{ij}^{\ttV}\alpha_i, &  &1\leq j\leq r.
\end{align*}
Since $s_i^{\ttV}(\alpha_i)=-\alpha_i$, we have that $(s_i^{\ttV})^2=\id_Q$.

\begin{lemma}\label{lem:reflected-datum-symmetrizable}
If $\ttV$ is $i$-finite, then the pair $(\rho^{(i)},\ttK)$ is a symmetrizable datum. 
\end{lemma}

\begin{proof} We use schematic notation. 
The summand $\tte_i^{(i)}=\ttf_i$ is a simple $\fh$-module with non-zero action by duality and condition \ref{item:symmetrizable-def-1} for the original datum. Fix $j\ne i$ and set $N=N_{ij}^{\ttV}$. By Lemma~\ref{lem:galphaj+nalphai-simple}, $\tte_j^{(i)}$ is a simple $\fh$-module. We claim that the $\fh$-action on $\tte_j^{(i)}$ is non-zero. For $N=0$, $\tte_j^{(i)}=\tte_j$ and the action is nontrivial by the original hypotheses. Suppose that $N>0$ and the $\fh$-action on $\tte_j^{(i)}$ is zero. Since $\tte_j^{(i)}$ is the top of the $i$-string, $[\tte_i,\tte_j^{(i)}]=0$. Set $Q_0\coloneqq[\ttf_i,\tte_j^{(i)}]\subseteq P_{ij,N-1}$, which is $\fh$-stable. Then  $[\tte_i,\ttf_i]\in\fh$, and the assumed trivial $\fh$-action on $\tte_j^{(i)}$ gives
\[
 [\tte_i,Q_0]
 \in  [[\tte_i,\ttf_i],\tte_j^{(i)}]+[\ttf_i,[\tte_i,\tte_j^{(i)}]]=0.
\]
Since $\fn_i^+$ is generated by $\tte_i$, it follows that $\fn_i^+$ annihilates $Q_0$. If $Q_0\ne0$, then Lemma~\ref{lem:submod-gen-from-V+} shows that $U(\fn_i^-)\cdot Q_0$ is a non-zero $\ttY_i$-submodule of
$M_{ij}$, which is contained in $\bigoplus_{0\leq r<N}P_{ij,r}$, and is therefore proper. This contradicts Lemma~\ref{lem:galphaj+nalphai-simple}. Hence $Q_0=0$.

It follows that $\tte_j^{(i)}$ is annihilated by $\fh$, $\tte_i$, and $\ttf_i$, so $\tte_j^{(i)}$ is a non-zero $\ttY_i$-submodule of $M_{ij}$. Again, as $M_{ij}$ is simple, $\tte_j^{(i)}=M_{ij}$. Then, $N=0$, which says that $\tte_j^{(i)}=\tte_j$ has trivial $\fh$-action, contradicting \ref{item:symmetrizable-def-1}. This proves the $\fh$-action on $\tte_j^{(i)}$ is non-zero.

Consequently, $ R_i(\ttV)$ is a finite direct sum of simple $\fh$-modules with nontrivial action. Then Remark \ref{rem:SD1-holds-ss} applies, so \ref{item:symmetrizable-def-1} holds.
Condition \ref{item:symmetrizable-def-2} is satisfied by hypothesis.
\end{proof}

Assume that $\ttV= (\rho,\ttd)$ is $i$-finite. 
Set 
$\ttd^{(i)}:=\ttd_{\rho^{(i)},\ttK}$, $R_i(\ttV) \coloneqq(\rho^{(i)},\ttd^{(i)})$.
Using the bilinear form we identify the dual of $\tte_j^{(i)}$ with the opposite homogeneous component,
\begin{align*}
(\tte_i^{(i)})^* & \cong\tte_i, & (\tte_j^{(i)})^*\cong P^-_{ij,N_{ij}^{\ttV}}=(\ad\ttf_i)^{N_{ij}^{\ttV}}(\ttf_j), \quad j\ne i.
\end{align*}

Thanks to the previous lemma we can consider the Lie algebras $\widetilde{\fg} \left(R_i(\ttV) \right)$ and $\fg \left(R_i(\ttV)\right)$. We have the following preliminary result.

\begin{lemma}\label{lem:Titilde}
If $\ttV$ is $i$-finite, there is a Lie algebra map
\begin{align}\label{eq:T_i def}
\widetilde{T}_i^{\,\ttV}: \widetilde{\fg}(R_i(\ttV))\longrightarrow\fg(\ttV),
\end{align}
which is the identity on $\fh$, the canonical inclusion for each $\tte_j^{(i)}$, and sends the dual $(\tte_j^{(i)})^*$ to the opposite component specified above. Moreover,
\[
 \widetilde{T}_i^{\,\ttV} \bigl(\widetilde{\fg}(R_i(\ttV))_{\beta}\bigr) \subseteq\fg(\ttV)_{s_i^{\ttV}(\beta)} \qquad\text{for all }\beta\in Q.
\]
\end{lemma}

\begin{proof}
The identifications $R_i(\ttV), R_i(\ttV)^* \hookrightarrow \fg(\ttV)$ are of $\fh$-modules, so they induce a map $\widehat{T}_i^{\,\ttV}:\FLie(R_i(\ttV)\oplus R_i(\ttV)^* )\rtimes_\rho \fh\to \fg(\ttV)$. It remains to check that the defining relations of $\widetilde{\fg}(R_i(\ttV))$ in \S \ref{subsec:contragredient-defs} belong to the kernel of this map. 

\begin{enumerate}[leftmargin=*,label=\rm{(\alph*)}]
\item If $j,k\ne i$ and $j\ne k$, then $s_i^{\ttV}(\alpha_j)-s_i^{\ttV}(\alpha_k)=\alpha_j-\alpha_k+(N_{ij}^{\ttV}-N_{ik}^{\ttV})\alpha_i$, so the corresponding homogeneous component of $\fg$ is zero. Then 
\begin{align*}
\widehat{T}_i^{\,\ttV}([\tte_j,\ttf_k])=\left[(\ad\tte_i)^{N_{ij}^{\ttV}}(\tte_j),(\ad\ttf_i)^{N_{ik}^{\ttV}}(\ttf_k)\right]=0.
\end{align*}
\item If $j=i$ and $k\ne i$, the degree of $\widehat{T}_i^{\,\ttV}([\tte_i,\ttf_k])$ is $-\alpha_k-(N_{ik}^{\ttV}+1)\alpha_i$, so the homogeneous component of this degree is zero by maximality of $N_{ik}^{\ttV}$. Then $\widehat{T}_i^{\,\ttV}([\tte_i,\ttf_k])=0$. The case $j\ne i$, $k=i$ is analogous.

\item For $j=k$, we take $u\in\tte_j^{(i)}$ and $v\in(\tte_j^{(i)})^*$. We have that
\begin{align*}
\ttK(x,[u,v]) &=\ttB(x,[u,v]) = \ttB([x,u],v) =\ev_{\tte_j^{(i)}}(\rho^{(i)}(x,u),v) & \text{for all }& x\in\fh.
\end{align*}
By the identity \eqref{eq:symmetrizable-dmap-characterization}, this says precisely that $[u,v]=\ttd^{(i)}(u,v)$.
\end{enumerate}
Thus all defining relations hold, and then $\widehat{T}_i^{\,\ttV}$ induces the desired map $\widetilde{T}_i^{\,\ttV}$. 
The statement about the grading follows from its values on the generators.
\end{proof}

We now arrive at the main result of this subsection, where we construct the desired reflections.

\begin{theorem}\label{thm:Ti-isomorphism}
If $\ttV$ is $i$-finite, the map \eqref{eq:T_i def} induces an isomorphism
\[
 T_i^{\ttV}: \fg(R_i(\ttV))\xrightarrow{\sim}\fg(\ttV)
\]
such that $T_i^{\ttV}\bigl(\fg(R_i(\ttV))_{\beta}\bigr)=\fg(\ttV)_{s_i^{\ttV}(\beta)}$ for all $\beta\in Q$.
\end{theorem}

\begin{proof}
The image of $\widetilde{T}_i^{\,\ttV}$ contains $\fh,\tte_i,\ttf_i$, and hence $\ttY_i$. For $j\ne i$, it also contains $P_{ij,N_{ij}^{\ttV}}$. By Lemma~\ref{lem:galphaj+nalphai-simple}, $P_{ij,N_{ij}^{\ttV}}$ generates the simple $\ttY_i$-module $M_{ij}$, so $\tte_j$ is also contained in the image. A dual argument proves that $\ttf_j\subseteq \ima \widetilde{T}_i^{\,\ttV}$. Hence $\ima \widetilde{T}_i^{\,\ttV}$ contains $\fh,\ttV,\ttV^*$, and then $\widetilde{T}_i^{\,\ttV}$ is surjective.

We consider a new grading on $\fg=\fg(\ttV)$ by setting
\begin{align*}
\fg^{(i)}_n&\coloneqq \bigoplus_{\he(\beta)=n}\fg_{s_i^{\ttV}(\beta)}, & &n\in\bZ.
\end{align*}
The map $\widetilde{T}_i^{\,\ttV}$ is surjective and homogeneous for this grading; in particular, its kernel is homogeneous and intersects $\fh$ trivially. The form $\ttB$
is non-degenerate when we restrict to $\fg^{(i)}_n \otimes \fg^{(i)}_{-n}$, since it is the direct sum of the non-degenerate pairings $\fg_{s_i^{\ttV}(\beta)}\otimes \fg_{-s_i^{\ttV}(\beta)}\longrightarrow\one$.
By Proposition~\ref{prop:contragredient-bilinear-form}, the kernel of $\widetilde{T}_i^{\,\ttV}$ is the maximal ideal defining $\fg(R_i(\ttV))$, so it induces the isomorphism $T_i^{\ttV}$.
The statement about the grading follows from  Lemma \ref{lem:Titilde}.
\end{proof}

\begin{corollary}\label{lem:reflection}
If $\ttV$ is $i$-finite, then $R_i(\ttV)$ is also $i$-finite and
\begin{align*}
c_{ij}^{R_i(\ttV)} &=c_{ij}^{\ttV} \quad \text{ for each }j\ne i,  & R_i^2(\ttV) &\cong\ttV.
\end{align*}
\end{corollary}

\begin{proof}
Put $N=N_{ij}^{\ttV}$. By Theorem~\ref{thm:Ti-isomorphism},
\[
 T_i^{\ttV}\bigl(\fg(R_i(\ttV))_{\alpha_j+n\alpha_i}\bigr)
 =\fg(\ttV)_{s_i^{\ttV}(\alpha_j+n\alpha_i)}
 =\fg(\ttV)_{\alpha_j+(N-n)\alpha_i}.
\]
For $n\geq0$, the last component is non-zero exactly when $0\leq n\leq N$. Thus the $i$-th string of $R_i(\ttV)$ has the same number of non-zero components as the $i$-th string of $\ttV$, so $c_{ij}^{R_i(\ttV)}=c_{ij}^{\ttV}$
and $R_i(\ttV)$ is $i$-finite by Lemma
\ref{lem:galphaj+nalphai-simple}   \ref{item:interval}. Its top component is sent to $\fg_{\alpha_j}=\tte_j$, while $(\ttf_i)^*\cong\tte_i$. Consequently, $R_i^2(\ttV)\cong\ttV$.
\end{proof}

\begin{remark}
After choosing the identification $R_i^2(\ttV)\cong\ttV$ and normalizing the duality identifications, we can choose the second reflection map to be the
inverse of $T_i^{\ttV}$, so
\begin{align*}
T_i^{\ttV}\circ T_i^{R_i(\ttV)} &=\id_{\fg(\ttV)},  & T_i^{R_i(\ttV)}\circ T_i^{\ttV}=\id_{\fg(R_i(\ttV))}.
\end{align*}
\end{remark}

\subsection{The Weyl groupoid and the root system}
\label{subsec:weyl-groupoid}

Our next task is to define the Weyl groupoid of a semisimple $\fh$-module $\ttV$. The reflections $R_i$ constructed in the previous subsection depend on the fixed decomposition of $\ttV$ as sum of simple submodules (and the enumeration of the summands). Hence, we will deal with Weyl groupoids of ordered tuples of simple modules, as in \cite[Chapter 12]{HS}.

Fix $r\in\bN$, and set $I=\{1,\ldots,r\}$, $Q=\bigoplus_{i\in I}\bZ\alpha_i$ as in the previous subsection, $\fh \in \Lie(\cC)$, $\ttK\colon \fh \otimes \fh \to \one$ a symmetric non-degenerate invariant form. 

\begin{definition}
An \emph{ordered presentation over $(\fh,\ttK)$} is a tuple
$(\tte_1,\ldots,\tte_r)$ of simple $\fh$-modules with non-trivial action. 
\end{definition}

We identify two ordered presentations when their corresponding entries are isomorphic as $\fh$-modules. 
Depending on the context, we may use the tuple notation $(\tte_1,\ldots,\tte_r)$ or the direct sum $\ttV=\tte_1\oplus\ldots\oplus\tte_r$.

Fix a class $x=[\ttV]$ and choose a representative $\ttV$. 
Let $\rho^{\ttV}$ denote the action of $\fh$ on $\ttV$. Then $(\rho^{\ttV},\ttK)$ is a symmetrizable pair by Remark \ref{rem:SD1-holds-ss}, so we can consider $\ttd^x=\ttd_{\rho^{\ttV},\ttK}$, and the contragredient Lie algebra
\begin{align*}
\fg^x &\coloneqq\fg(\rho^x,\ttd^x), & \fg^x=\bigoplus_{\beta\in Q}\fg^x_\beta,
\end{align*}
which is well defined up to $Q$-graded isomorphism of contragredient Lie algebras. Consequently, the support, the integers $N_{ij}^x$, and all constructions below are independent of the chosen representative. Following the notation in \cite{AA-super}, let
\begin{align}
\varDelta_{\pm}^x &\coloneqq \{\beta\in Q^{\pm}\setminus\{0\}:\fg^x_\beta\ne0\}, & \varDelta^x & \coloneqq\varDelta_+^x\cup\varDelta_-^x,
\label{eq:roots-contragredient}
\end{align}
The non-degenerate pairings $\fg^x_\beta\otimes\fg^x_{-\beta}\to\one$ imply $\varDelta_-^x=-\varDelta_+^x$. 
Whenever $x$ is $i$-finite, we define $R_i(x)\coloneqq[R_i(\ttV)]$, again well-defined on isomorphism classes.

\begin{definition}
Let $\cX(\ttV)$ be the set of classes obtainable from $[\ttV]$ by a finite sequence of reflections $R_i$. We say that $\ttV$ is \emph{regular} if every $x\in\cX(\ttV)$ is $i$-finite for every $i\in I$, i.e. if $N_{ij}^x$ is finite for all $i\ne j\in I$ and all $x\in\cX(\ttV)$. 
\end{definition}

Hence, for a regular $\ttV$, the maps $R_i\colon\cX(\ttV)\longrightarrow\cX(\ttV)$ are well defined for all $i\in I$ as well as the matrices $C^x=(c_{ij}^x)_{i,j\in I}$ as defined in \eqref{eq:defn-cij}, $x\in\cX(\ttV)$.

A particular case of a regular $\ttV$ is when $\fg^{[\ttV]}\in\cC$.

\begin{definition}
The \emph{Weyl groupoid} of $\ttV$ is the subgroupoid $\cW(\ttV)$ of $\cX(\ttV)\times \operatorname{GL}(Q)\times\cX(\ttV)$ generated by
\begin{align*}
\sigma_i^x & \coloneqq (x, s_i^x, R_i(x)), & i\in I, \, & x\in\cX(\ttV).
\end{align*}
\end{definition}

We regard $(x,w,y)$ as a morphism $y\to x$, with composition
\[
(x,w,y)\circ(y,v,z)=(x,wv,z).
\]
Note that the inverse of $\sigma_i^x$ is $\sigma_i^{R_i(x)}=(R_i(x), s_i^x, x)$, see Corollary~\ref{lem:reflection}.

Next we introduce the set of real roots of $x\in\cX(\ttV)$, in analogy with the classical case:
\begin{align*}
\varDelta^{x,\re}& \coloneqq \{ \omega(\alpha_i) \, | \, i\in I, \, (x,\omega,y)\in\cW(\ttV), \, y\in\cX(\ttV) \}.
\end{align*}
That is, the set of real roots contains all the elements of $Q$ obtained after applying compositions of $\sigma_i^y$, ending at $x$.
The subsets of positive, respectively negative, real roots are $\varDelta^{x,\re}_{\pm}=\varDelta^{x,\re}\cap(\pm \bN_0^I)$, and
\begin{align*}
m_{ij}^{x} \coloneqq | \varDelta^{x,\re} \cap (\bN_0\alpha_i+\bN_0\alpha_j) |. 
\end{align*}

Now we are ready to check that all the previous data give rise to a root system in the sense of \cite{HS}.

\begin{theorem}\label{th:cartan-graph-root-system}
Let $\ttV$ be regular. 

\begin{enumerate}[leftmargin=*,label=\rm{(\alph*)}]

\item\label{item:cartan-graph} $\cG(\ttV)=\bigl(I,\cX(\ttV),(R_i)_{i\in I},(C^x)_{x\in\cX(\ttV)}\bigr)$ is a Cartan graph, cf. \cite[Definition 9.1.14]{HS}. More precisely,

\begin{enumerate}[leftmargin=*,label=\rm{CG\arabic*)}]
\item\label{item:cartan-graph1} $R_i^2=\id_{\cX(\ttV)}$ for every $i \in I$;
\item\label{item:cartan-graph2} each $C^x$ is a generalized Cartan matrix and $c_{ij}^x=c_{ij}^{R_i(x)}$ for all $x\in\cX(\ttV)$ and $i,j\in I$.
\item\label{item:cartan-graph3} $\varDelta^{x,\re}=\varDelta^{x,\re}_+\cup \varDelta^{x,\re}_-$ for all $x\in\cX(\ttV)$;
\item\label{item:cartan-graph4} for each $x\in\cX(\ttV)$ and $i \ne j\in I$ such that $m_{ij}^x<\infty$, $(R_iR_j)^{m_{ij}^x}(x)=x$.
\end{enumerate}

\item\label{item:root-system} $\cR(\ttV)=\bigl(\cG(\ttV),(\varDelta^x)_{x\in\cX(\ttV)}\bigr)$ is a root system, cf. \cite[Definition 10.4.1]{HS}. More precisely:
\begin{enumerate}[leftmargin=*,label=\rm{RS\arabic*)}]
\item\label{item:root-system1} $0\notin \varDelta^x$, $\alpha_i\in\varDelta^x$ for all $i\in I$, $x\in\cX(\ttV)$.
\item\label{item:root-system2} $\varDelta^x\subset \bN_0^I \cup (-\bN_0^I)$ for all $x\in\cX(\ttV)$.
\item\label{item:root-system3} $s_i^x(\varDelta^x)=\varDelta^{R_i(x)}$ for all $i\in I$, $x\in\cX(\ttV)$.
\end{enumerate}

\end{enumerate}
Consequently, $\cW(\ttV)$ is the Weyl groupoid of the contragredient datum $\ttV$.
\end{theorem}

\begin{proof}
We prove \ref{item:root-system} first. Property  \ref{item:root-system1} is clear, and \ref{item:root-system2} follows from the triangular decomposition of $\fg^x$. Lastly,  \ref{item:root-system3} is a consequence of the isomorphism $T_i^x$ in Theorem \ref{thm:Ti-isomorphism}.

Now we check \ref{item:cartan-graph}.  
We first check that $\cG(\ttV)$ is a semi-Cartan graph, cf. \cite[Definition 9.1.1]{HS}; that is, it satisfies \ref{item:cartan-graph1} and \ref{item:cartan-graph2}. The equality $R_i^2=\id$ and the equality of the $i$-th rows of $C^x$ and $C^{R_i(x)}$ follow from Corollary~\ref{lem:reflection}. For $i\ne j$, Lemma~\ref{lem:galphaj+nalphai-simple} gives
\begin{align*}
c_{ij}^x&=0 && \Longleftrightarrow & \fg^x_{\alpha_i+\alpha_j}&=0 && \Longleftrightarrow & c_{ji}^x&=0.
\end{align*}
Thus $C^x$ is a generalized Cartan matrix for all $x\in\cX(\ttV)$.

Property \ref{item:cartan-graph3} is clear, as $\varDelta^{x,\re}\subseteq \varDelta^{x}$ and we already checked that \ref{item:root-system2} holds. It remains to prove \ref{item:cartan-graph4}. 
Fix $i\ne j$, put $m=m_{ij}^x<\infty$, and define
\[
x_0=x,\qquad x_{n+1}=R_iR_j(x_n),\qquad
w_0=\id_Q,\qquad
w_{n+1}=s_i^{R_j(x_n)}s_j^{x_n}w_n.
\]
Then the properties R1--R3 used in the proof of \cite[Theorem~6.11]{HS-PLMS} are satisfied, and notice that the rank-two argument in that proof, based on the proof of \cite[Lemma~5]{HY}, uses only these properties. Hence, following their argument, 
$w_m(\alpha_i)=\alpha_i$ and $w_m(\alpha_j)=\alpha_j.$ 
For $k\notin\{i,j\}$ write
$w_m(\alpha_k)=\alpha_k+a\alpha_i+b\alpha_j$.
This belongs to $\varDelta^{x_m}$, so RS2 gives $a,b\geq0$. Since $w_m$ fixes $\alpha_i,\alpha_j$ and $w_m^{-1}(\varDelta^{x_m})=\varDelta^x$, we also have
\[
w_m^{-1}(\alpha_k)=\alpha_k-a\alpha_i-b\alpha_j\in\varDelta^x.
\]
RS2 gives $a=b=0$, hence $w_m=\id_Q$.

The inverse reflection isomorphisms compose to an isomorphism
$\Psi:\fg^x\to\fg^{x_m}$ which is the identity on $\fh$ and sends
$\fg^x_\beta$ to $\fg^{x_m}_{w_m(\beta)}$.
Therefore $\Psi$ identifies the degree-$\alpha_k$ components as $\fh$-modules for every $k$. The corresponding ordered presentations are isomorphic, so $x_m=x$, which is CG4.
\end{proof}


As in the classical setting, we say that the Cartan graph $\cG(\ttV)$ is \emph{finite} if $\varDelta^{x,\re}$ is finite for all $x\in\cX(\ttV)$. Also, the root system $\cR(\ttV)$ is \emph{reduced} if for all $x\in\cX(\ttV)$ and all $\alpha\in\varDelta^x$, $\bQ\alpha\cap \varDelta^x=\{\pm\alpha\}$. We can construct a reduced root system $\cR^{\operatorname{red}}(\ttV)$ by considering
\begin{align*}
\varDelta^{x,\operatorname{red}} &\coloneqq \{\alpha\in\varDelta^x \, | \, q\alpha\notin \varDelta^x \text{ for all }q\in(0,1)\}, & &x\in\cX(\ttV).
\end{align*}

\begin{corollary}\label{cor:finite-categorical-weyl-groupoid}
If $\fg(\ttV)\in\cC$, then $\ttV$ is regular. The root system $\mathcal R(\ttV)$ and its Weyl groupoid $\mathcal W(\ttV)$ are finite. In particular, $\varDelta^{\ttV,\operatorname{red}}=\varDelta^{\ttV,\re}$.
\end{corollary}

\begin{proof}
As observed before, $\ttV$ is regular. As $\varDelta^x$ is finite for all $x\in\cX(\ttV)$, the corresponding reduced root system is univocally determined as the root systems of real roots by \cite[Theorem 10.4.7]{HS}.
\end{proof}

\begin{remark}\label{rem:not-reduced-root-systems}
The root system of a contragredient Lie algebra need not be reduced. For example, $\mathfrak{osp}(1|2)$ has root system $\{\pm\alpha,\pm2\alpha\}$. Furthermore, the Lie algebras in $\Verp$, $p>3$, obtained by semisimplification of the simple ones of type $B_2$ and $G_2$ with respect to $e_2$ in \cite[Examples 4.33, 4.34]{APS-ss} are contragredient Lie algebras whose root system is $\{\pm \alpha_1,\pm 2\alpha_1\}$, respectively $\{\pm \alpha_1,\pm 2\alpha_1, \pm 3\alpha_1\}$.
\end{remark}

\subsection{An example in \texorpdfstring{$\mathrm{Ver}_5$}{Ver5}}
\label{subsec:example-weyl-gpd}

We conclude by illustrating the constructions of this section with an example.  The point is to see explicitly how the homogeneous
components determine the Cartan matrix, how reflections change the
ordered presentation without changing the Lie algebra, and how the
resulting presentations assemble into a Weyl groupoid.

Assume that $\operatorname{char}\Bbbk=5$ and let
$\cC=\mathrm{Ver}_5$.  Set
\[
 U_1=\one,\qquad U_2=L_2,\qquad
 U=U_1\oplus U_2,\qquad W=U\oplus U^*.
\]
Equip $W$ with the canonical non-degenerate skew form for which $U$ and
$U^*$ are isotropic, and let
\[
 \fg=\mathfrak{sp}(W).
\]

\smallskip
\noindent\emph{The contragredient presentation.}
The block-diagonal subalgebra
\[
 \fh=\mathfrak{gl}(U_1)\oplus\mathfrak{gl}(U_2)
 =\mathfrak{sl}(L_2)\oplus\Bbbk t_1\oplus\Bbbk t_2
\]
acts on $W$ by $x\mapsto x\oplus(-x^*)$.  Here $t_a$ acts as the
identity on $U_a$ and as minus the identity on $U_a^*$.  Let
$\varepsilon_a$ record the eigenvalue of $t_a$, and put
\begin{align*}
\alpha_1&:=\varepsilon_1-\varepsilon_2,
&
\alpha_2&:=2\varepsilon_2,
&
Q&:=\mathbb Z\alpha_1\oplus\mathbb Z\alpha_2.
\end{align*}

Let $T=(\mathbb G_m)^2$ act on $W$ by $z_a$ on $U_a$ and by
$z_a^{-1}$ on $U_a^*$.  This action preserves the skew form and hence
induces an action of $T$ on $\fg$ by conjugation.  The differential of
the $a$-th factor of $T$ is the element $t_a\in\fh$; thus $t_a$ acts as
the identity on $U_a$ and as minus the identity on $U_a^*$.
Let $\varepsilon_1,\varepsilon_2\in X^*(T)$ be the standard integral
characters, and put
\begin{align*}
 \alpha_1&:=\varepsilon_1-\varepsilon_2,
 &
 \alpha_2&:=2\varepsilon_2,
 &
 Q&:=\mathbb Z\alpha_1\oplus\mathbb Z\alpha_2.
\end{align*}
Then the
$T$-weight decomposition defines a $Q$-grading
$  \fg=\bigoplus_{\beta\in Q}\fg_\beta.$ 
For $y\in\fg_\beta$, differentiating the $T$-action gives
$  [t_a,y]=\beta(t_a)y$ for  $a=1,2,$ 
where the integer $\beta(t_a)$ is viewed as an element of $\Bbbk$.

Since $T$ centralizes $\fh$, every homogeneous component
$\fg_\beta$ is an $\fh$-submodule.  We take $\alpha_1,\alpha_2$ to be the simple positive
degrees.
The positive homogeneous components of $\fg$ are
\[
\begin{array}{c|c}
 \beta & \fg_\beta 
 \\
 \hline
 \alpha_1 & U_1\otimes U_2^*\cong L_2
   \\
 \alpha_2 & S^2U_2\cong L_3
   \\
 \alpha_1+\alpha_2 & U_1\otimes U_2\cong L_2
    \\
 2\alpha_1+\alpha_2 & S^2U_1\cong\one
\end{array}
\]
Moreover, $\fg_0=\fh$ and $\fg_{-\beta}\cong\fg_\beta^*$.  Thus
\[
 \ttV=\tte_1\oplus\tte_2
 :=\fg_{\alpha_1}\oplus\fg_{\alpha_2}
 \cong L_2\oplus L_3
\]
is a sum of two simple $\fh$-modules with nontrivial action.

Let $\ttK$ be the restriction to $\fh$ of the trace form on
$\mathfrak{gl}(W)$.  Its restriction to each block
$\mathfrak{gl}(U_a)$ is twice the canonical trace pairing, and is
therefore non-degenerate.  Moreover, $\fh$ acts faithfully on $W$, so
it has enough modules.  Hence $(\rho,\ttK)$, where $\rho$ is the
adjoint action of $\fh$ on $\ttV$, is a symmetrizable datum.  Under the
identification of $\ttV^*$ with
$\fg_{-\alpha_1}\oplus\fg_{-\alpha_2}$ induced by the trace form, the
map $\ttd_{\rho,\ttK}$ is the bracket.

The remaining positive root spaces are generated successively:
\[
 [\fg_{\alpha_1},\fg_{\alpha_2}]
   =\fg_{\alpha_1+\alpha_2},
 \qquad
 [\fg_{\alpha_1},\fg_{\alpha_1+\alpha_2}]
   =\fg_{2\alpha_1+\alpha_2}.
\]
Indeed, these brackets are the usual contraction
$L_2^*\otimes S^2L_2\to L_2$ and evaluation
$L_2^*\otimes L_2\to\one$, and both are nonzero.  The negative
components are generated dually.  The trace form restricts to a
non-degenerate invariant form on $\fg$: indeed, $\fg$ is the
$1$-eigenspace of the trace-form-preserving involution
$x\mapsto-x^\dagger$, where $\dagger$ denotes the adjoint with respect
to the skew form, and the two eigenspaces are orthogonal.  Therefore
Proposition~\ref{prop:contragredient-bilinear-form} gives
\[
 \fg\cong\fg(\rho,\ttd_{\rho,\ttK}).
\]

\smallskip
Notice that $N_{12}=2$ and $N_{21}=1$, so the Cartan matrix is $C^{\ttV}=
 \begin{pmatrix}
  2&-2\\
  -1&2
 \end{pmatrix}$, of type $C_2$.

\smallskip
\noindent\emph{Reflections.}
Let $x_0=[(\tte_1,\tte_2)]$ be the class of the presentation above.
The reflection at the first vertex replaces $\tte_1$ by its dual and
$\tte_2$ by the top of its $\alpha_1$-string:
\[
 R_1(x_0)
 =
 \left[
 \left(
 \ttf_1,\,
 (\operatorname{ad}\tte_1)^2(\tte_2)
 \right)
 \right]
 =
 \left[
 \left(
 \fg_{-\alpha_1},\,
 \fg_{2\alpha_1+\alpha_2}
 \right)
 \right].
\]
In terms of the ordered pair $(U_1,U_2)$, this reflection interchanges
the two entries.  Similarly,
\[
 R_2(x_0)
 =
 \left[
 \left(
 [\tte_2,\tte_1],\,
 \ttf_2
 \right)
 \right]
 =
 \left[
 \left(
 \fg_{\alpha_1+\alpha_2},\,
 \fg_{-\alpha_2}
 \right)
 \right],
\]
which replaces $U_2$ by $U_2^*$.

More generally, let
\[
 W(C_2)\cong(\mathbb Z/2\mathbb Z)^2\rtimes S_2
\]
be the Weyl group of type $C_2$.  It acts on
$\varepsilon_1,\varepsilon_2$ by signed permutations: its elements may
permute the two weights and independently change their signs.  Under the
correspondence with ordered pairs, permutations interchange the objects
$U_1,U_2$, while changing the sign of $\varepsilon_a$ replaces $U_a$ by
$U_a^*$.
  If
$w(\varepsilon_i)=\eta_i\varepsilon_{\pi(i)}$, set
\[
 U'_i=U_{\pi(i)}^{\eta_i},
 \qquad
 U_a^+=U_a,\quad U_a^-=U_a^*,
\]
and define
\[
 x_w=
 \left[
 \left(
 U'_1\otimes {U'_2}^*,\,S^2U'_2
 \right)
 \right]
 =
 \left[
 \left(
 \fg_{w(\alpha_1)},\,\fg_{w(\alpha_2)}
 \right)
 \right].
\]
The preceding computations give
\[
 R_i(x_w)=x_{ws_i},
\]
where $s_1$ interchanges $\varepsilon_1$ and $\varepsilon_2$, while
$s_2$ changes the sign of $\varepsilon_2$.  The elements $t_1,t_2$
distinguish the eight presentations $x_w$: their eigenvalues on the
root spaces lie in $\{0,\pm1,\pm2\}$, which are distinct in
characteristic $5$.  Thus
\[
 W(C_2)\longrightarrow\cX(\ttV),\qquad w\longmapsto x_w,
\]
is a bijection.

Although the eight presentations are distinct as $\fh$-modules, they
have only two shapes as objects of $\cC$:
\[
\begin{array}{c|c|c|c|c|c}
 & (U'_1,U'_2)
 & (\fg^{x_w})_{\alpha_1}
 & (\fg^{x_w})_{\alpha_2}
 & (\fg^{x_w})_{\alpha_1+\alpha_2}
 & (\fg^{x_w})_{2\alpha_1+\alpha_2}\\
 \hline
 \mathrm A
 & (\one^\pm,L_2^\pm)
 & L_2&L_3&L_2&\one\\
 \mathrm B
 & (L_2^\pm,\one^\pm)
 & L_2&\one&L_2&L_3.
\end{array}
\]
Here the signs retain the $\fh$-module structure even though
$\one^*\cong\one$ and $L_2^*\cong L_2$ as objects of $\cC$.

The reflection $R_1$ exchanges the two shapes, whereas $R_2$ preserves
each shape.  The exchange graph is the following octagon:
\begin{center}
\begin{tikzpicture}[
  x=1cm,y=1cm,
  every node/.style={font=\scriptsize},
  vertex/.style={draw,rounded corners,fill=white,inner sep=2pt},
  edge label/.style={fill=white,inner sep=1pt}
]
\node[vertex] (x0) at (0,2.3) {$(\one,L_2)$};
\node[vertex] (x1) at (2.7,1.6) {$(\one,L_2^*)$};
\node[vertex] (x2) at (3.8,0) {$(L_2^*,\one)$};
\node[vertex] (x3) at (2.7,-1.6) {$(L_2^*,\one^*)$};
\node[vertex] (x4) at (0,-2.3) {$(\one^*,L_2^*)$};
\node[vertex] (x5) at (-2.7,-1.6) {$(\one^*,L_2)$};
\node[vertex] (x6) at (-3.8,0) {$(L_2,\one^*)$};
\node[vertex] (x7) at (-2.7,1.6) {$(L_2,\one)$};
\draw (x0)--node[edge label] {$R_2$} (x1);
\draw (x1)--node[edge label] {$R_1$} (x2);
\draw (x2)--node[edge label] {$R_2$} (x3);
\draw (x3)--node[edge label] {$R_1$} (x4);
\draw (x4)--node[edge label] {$R_2$} (x5);
\draw (x5)--node[edge label] {$R_1$} (x6);
\draw (x6)--node[edge label] {$R_2$} (x7);
\draw (x7)--node[edge label] {$R_1$} (x0);
\end{tikzpicture}
\end{center}

\smallskip
\noindent\emph{The root system and Weyl groupoid.}
For every $w\in W(C_2)$, Theorem~\ref{thm:Ti-isomorphism} identifies
$\fg^{x_w}$ with the same Lie algebra $\fg$, but with the grading
changed by $w$:
\[
 (\fg^{x_w})_\beta\cong\fg_{w(\beta)}.
\]
Therefore every object has the same reduced root
system
\[
\varDelta^{x_w}=\varDelta^{x_w,\re}
 =
 \left\{
 \pm\alpha_1,\ \pm\alpha_2,\ \pm(\alpha_1+\alpha_2),\
 \pm(2\alpha_1+\alpha_2)
 \right\}.
\]
In particular, $m_{12}^{x_w}=4$ and
$(R_1R_2)^4(x_w)=x_w$, as predicted by
Theorem~\ref{th:cartan-graph-root-system}.  The Weyl groupoid has eight
objects, its exchange graph is the octagon above, and all isotropy
groups are trivial.  Its reflections act on $Q$ as the usual
generators of the dihedral group $W(C_2)$ of order $8$.

\smallskip
\noindent\emph{Simplicity.}
We finish by applying the criterion of
Theorem~\ref{thm:contragredient-indec-simple}.  Let
$\tth_i=[\tte_i,\ttf_i]=\operatorname{im}\ttd_i$.  From the defining
property of $\ttd_{\rho,\ttK}$,
$  \tth_i=(\Ann_{\fh}\tte_i)^\perp.$ 
In the present example,
\[
\Ann_{\fh}\tte_1=\Bbbk(t_1+t_2),
 \qquad
 \Ann_{\fh}\tte_2=\mathfrak{gl}(U_1).
\]
Indeed, the two scalar blocks act on
$\tte_1=U_1\otimes U_2^*$ through their difference, while
$\mathfrak{gl}(U_2)$ acts faithfully on
$\tte_2=S^2L_2\cong L_3$.  It follows that
\[
 \tth_1+\tth_2=\fh,
 \qquad
 \Ann\ttV
 =
 \Ann_{\fh}\tte_1\cap\Ann_{\fh}\tte_2
 =0.
\]
Thus $\fg'=\fg$, and Lemma~\ref{lem:center-contragredient} gives
$\fc(\fg)=0$.  Moreover, $\tth_2$ acts nontrivially on $\tte_1$, while
$\tth_1$ acts nontrivially on $\tte_2$, so the datum is
indecomposable.  Note that $\fg'_0=\fh$ is generated by $\tth_1$ and $\tth_2$, and the generating modules $\tte_1\cong L_2$ and $\tte_2\cong L_3$ are simple as objects of $\Ver_5$, hence also as $\fg'_0$-modules. By Theorem~\ref{thm:contragredient-indec-simple}, since the connected quiver has arrows in both directions and $\Ann\ttV=0$, then $\fg$ is simple.

\section*{Acknowledgments}

This work was supported by the American Institute of Mathematics through its
SQuaREs program.  I. A. was partially supported by CONICET and Secyt (UNC). The work of P. E. was partially supported by the NSF grant DMS-2502467.
A.C. was partially supported by Simons Collaboration Grant No. 999367. J.P. was partially supported by US NSF Grant DMS2146392 and by Simons Foundation Award
SFI-MPS-GCS-00008528-08 as part of the Simons Collaboration on Global Categorical Symmetries. 

The authors used OpenAI's ChatGPT 5.6 Pro.  It assisted with some of the theoretical development. It was also used to help develop illustrative examples, check computations and the logical consistency of arguments, and suggest possible revisions. All suggestions were critically evaluated by the authors, and all arguments, computations, and references in the final manuscript were independently verified. The authors take full responsibility for the content of the paper.

\end{document}